\documentclass[final]{amsart}

\usepackage{graphicx}
\usepackage{amsmath,amsthm,amsfonts,amscd,amssymb}
\usepackage[color,notref,notcite]{showkeys}
\usepackage{url}
\usepackage{subeqnarray}
\usepackage[color,notref,notcite]{showkeys} 
\definecolor{labelkey}{rgb}{0.6,0,1}

\theoremstyle{plain}
\newtheorem{theorem}{Theorem}[section]

\newtheorem{hypothesis}[theorem]{Hypothesis}

\theoremstyle{definition}
\newtheorem{definition}[theorem]{Definition}

\def\bhyp#1{\begin{equation}\label{#1}\begin{array}{c}}
\def\ehyp{\end{array}\end{equation}}

\newcounter{cst}

\theoremstyle{remark}
\newtheorem{remark}[theorem]{Remark}

\numberwithin{equation}{section}
\numberwithin{figure}{section}

\newcommand{\RR}{{\mathbb R}}
\newcommand{\NN}{{\mathbb N}}

\def\O{\Omega}
\def\dsp{\displaystyle}
\def\bfn{\mathbf{n}}
\def\disc{{\mathcal D}}
\def\mesh{{\mathcal M}}
\newcommand{\polyd}{{\mathcal T}}
\def\edges{{\mathcal E}}
\def\edge{\sigma}
\def\xcv{x_K}
\def\cv{K}

\newcommand{\edgescv}{{{\edges}_K}}  
\newcommand{\edgesext}{{{\edges}_{\rm ext}}} 
\newcommand{\edgesint}{{{\edges}_{\rm int}}} 
\newcommand{\centers}{\mathcal{P}}

\def\dr{\partial}
\newcommand{\centeredge}{\overline{x}_\edge} 

\newif\ifcorr\corrtrue
\usepackage[normalem]{ulem}
\definecolor{violet}{rgb}{0.580,0.,0.827}

\def\bpsi{{\boldsymbol \psi}}

\newcommand{\ud}{\, \mathrm{d}} 
\def\div{{\rm div}}

\title[Gradient schemes for the Signorini and the obstacle problems]{Gradient schemes for the Signorini and the obstacle problems, and application to hybrid mimetic mixed methods}

\author{Yahya Alnashri}
\address[Yahya Alnashri]{School of Mathematical Sciences, Monash University, Victoria 3800, Australia\\
and Umm-Alqura University}
\email{yahya.alnashri@monash.edu\\
yanashri@ummalqura.edu}

\author{J\'er\^ome Droniou}
\address[J\'er\^ome Droniou]{School of Mathematical Sciences, Monash University, Victoria 3800, Australia.}
\email{jerome.droniou@monash.edu}



\date{\today}

\begin{document}


\begin{abstract}
Gradient schemes is a framework which enables the unified convergence analysis of many different methods -- such as finite elements (conforming, non-conforming and mixed) and finite volumes methods -- for $2^{\rm nd}$ order diffusion equations. We show in this work that the gradient schemes framework can be extended to variational inequalities involving mixed Dirichlet, Neumann and Signorini boundary conditions. 
This extension allows us to provide error estimates for numerical approximations of such models,
recovering known convergence rates for some methods, and establishing new convergence rates for
schemes not previously studied for variational inequalities. The general framework we develop
also enables us to design a new numerical method for the obstacle and Signorini problems,
based on hybrid mimetic mixed schemes. We provide
numerical results that demonstrate the accuracy of these schemes, and confirm our theoretical rates
of convergence.
\end{abstract}

\maketitle



\section{Introduction}
\label{introduction}
\noindent Numerous problems arising in fluid dynamics, elasticity, biomathematics, mathematical economics and 
control theory are modelled by second order partial differential equations with one-sided constraints on the solution or
its normal derivative on the boundary. Such models also appear in
free boundary problems, in which partial differential equations (PDEs) are written on a domain whose boundary is not given but has to be located as a part of the solution. From the mathematical point of view, these models
can be recast into variational inequalities \cite{G1,G12,G13}.
Linear variational inequalities are used in the study of
triple points in electrochemical reactions \cite{KDJPHM}, of
contact in linear elasticity \cite{hild2007}, and of lubrication phenomena
\cite{capriz1977}. Other models involve non-linear variational inequalities;
this is for example the case of unconfined seepage models, which involve
non-linear quasi-variational inequalities obtained through the Baiocchi transform \cite{baiocchi},
or quasi-linear classical variation inequalities through the extension of the Darcy
velocity inside the dry domain \cite{A1,Desai.Li:83,A25,A26}.
It is worth mentioning that several of these linear and non-linear models involve possibly heterogeneous and
anisotropic tensors. For linear equations, this happens in models of contact elasticity problems involving
composite materials (for which the stiffness tensor depends on the position),
and in lubrication problems (the tensor is a function of the first fundamental form of the film).
For non-linear models, heterogeneity and anisotropy are encountered in
seepage problems in porous media.

In this work, we consider two types of linear elliptic variational inequalities: the Signorini problem
and the obstacle problem. The Signorini problem is formulated as
\begin{align}
-\mathrm{div}(\Lambda\nabla\bar{u})= f &\mbox{\quad in $\Omega$,} \label{seepage1}\\
\bar{u} =0 &\mbox{\quad on $\Gamma_{1}$,} \label{seepage2}\\
\Lambda\nabla\bar{u}\cdot \mathbf{n}= 0 &\mbox{\quad on $\Gamma_{2}$,} \label{seepage3}\\
\left.
\begin{array}{r}
\dsp \bar{u} \leq a\\
\dsp \Lambda\nabla\bar{u}\cdot \mathbf{n} \leq 0\\
\dsp \Lambda\nabla\bar{u}\cdot \mathbf{n} (a-\bar{u}) = 0
\end{array} \right\} 
& \mbox{\quad on}\; \Gamma_{3}.
\label{sigcondition}
\end{align}
The obstacle problem is
\begin{align}
(\mathrm{div}(\Lambda\nabla\bar{u})+f)(g-\bar{u}) &= 0 \mbox{\quad in $\Omega$,} \label{obs1}\\
-\mathrm{div}(\Lambda\nabla\bar{u}) &\leq f \mbox{\quad in $\Omega$,} \label{obs2}\\
\bar{u}&\leq g \mbox{\quad in $\Omega$,} \label{obs3}\\
\bar{u} &= 0 \mbox{\quad on $\partial\Omega$.} \label{obs4}
\end{align}
Here $\Omega$ is a bounded open set of $\RR^d$ ($d\ge 1$),
$\mathbf{n}$ is the unit outer normal to $\partial\Omega$ and $(\Gamma_1,\Gamma_2,\Gamma_3)$
is a partition of $\partial\Omega$ (precise assumptions are stated in the next section).

Mathematical theories associated with the existence, uniqueness and stability of the solutions to variational inequalities have been extensively developed \cite{G10,G12,G1}.
From the numerical point of view, different methods have been considered to
approximate variational inequalities. Bardet and Tobita \cite{A9} applied a finite difference scheme to the unconfined seepage problem. Extensions of the discontinuous Galerkin method to solve the obstacle problem can be found in \cite{A13,A14}. Although this method is still applicable when the functions are discontinuous along the elements boundaries, the exact solution has to be in the space $H^{2}$ to ensure the consistency and the convergence. Numerical verification methods, which aim at finding a set in which a solution exists, have been developed for a few variational inequality problems. In particular we cite the obstacle problems, the Signorini problem and elasto-plastic problems (see \cite{A19} and references therein).

Falk \cite{F1} was the first to provide a general error estimate for the approximation by conforming methods of the obstacle problem. This estimate showed that the $\mathbb{P}_1$ finite element error is $\mathcal{O}(h)$. Yang and Chao \cite{A28} showed that the convergence rate of non-conforming finite elements method for the Signorini problem has order one, under an $H^{5/2}$-regularity assumption on the solution.

Herbin and Marchand \cite{A2} showed that if $\Lambda \equiv I_{d}$ and if the grid
satisfies the orthogonality condition required by the two-point flux approximation (TPFA) finite volume
method, then for both problems the solutions provided by this scheme converge in $L^{2}(\Omega)$ to the unique solution as the mesh size tends to zero.

The gradient schemes framework is a discretisation of weak variational formulations of PDEs using a small number of discrete elements and properties. 
Previous studies \cite{B10} showed that this framework includes many well-known numerical schemes:
conforming, non-conforming and mixed finite elements methods (including the non-conforming ``Crouzeix--Raviart'' method and the  Raviart--Thomas method),
hybrid mimetic mixed methods (which contain hybrid mimetic finite differences,
hybrid finite volumes/SUSHI scheme and mixed finite volumes), nodal mimetic finite differences,
and finite volumes methods (such as some multi-points flux approximation and discrete duality finite volume methods). Boundary conditions can be considered in various
ways within the gradient schemes framework, which is a useful flexibility for
some problems and engineering constraints (e.g. available meshes, etc.).
This framework has been analysed for several linear and non-linear elliptic and parabolic problems, including the Leary-Lions, Stokes, Richards, Stefan's and elasticity equations. It is noted that, for models based on linear operators as in \eqref{seepage1} only three properties are sufficient to obtain the convergence of schemes for the corresponding PDEs. We refer the reader to
\cite{B1,B2,B9,S1,DE-fvca7,zamm2013,eym-12-stef} for more details.

The aim of this paper is to develop a gradient schemes framework for elliptic linear variational inequalities with different types of boundary conditions. Although we focus here on scalar models, as opposed to variational inequalities
involving displacements, we notice that using \cite{B9} our results can be easily adapted to elasticity models.
The gradient scheme framework enables us to design a unified convergence analysis of many numerical methods for variational inequalities; using the theorems stemming from this analysis,
we recover known estimates for some schemes, and we establish new estimates for methods which were not, to our best
knowledge, previously studied for variational inequalities. In particular, we apply our results
to the hybrid mimetic mixed (HMM) method of \cite{B5}, which contains the mixed/hybrid mimetic finite difference methods and, on certain grids (such as triangular grids), the TPFA scheme.
We would like to point out that, to our best knowledge, although several papers studied
the obstacle problem for \emph{nodal} mimetic finite differences (see e.g. \cite{A-32,MFD2}),
no study has been conducted of the mixed/hybrid mimetic method for the obstacle problem,
or of any mimetic scheme for the Signorini problem. Our analysis through the gradient
schemes framework, and our application to the HMM method therefore seems to provide
entirely new results for these schemes on meaningful variational inequality models.
We also mention that we are mostly interested in low-order methods, since the models we consider could involve discontinuous data -- as in heterogeneous media -- which prevent the solution from being smooth. Finally, it seems to us that our unified analysis gives simpler proofs than the previous analyses in the literature. The study of linear models is a first step towards a complete analysis of non-linear models, a work already initiated in \cite{Y}.

This paper is organised as follows. In Section \ref{main result}, we present the gradient
schemes framework for variational inequalities, and we state our main error estimates. We then show in Section \ref{sec:examples} that several methods, of practical relevance
for anisotropic heterogeneous elliptic operators, fit into this framework
and therefore satisfy the error estimates established in Section \ref{main result}.
This also enables us to design and analyse for the first time an HMM scheme for
variational inequalities. We provide in Section \ref{sec:numer} numerical results that demonstrate
the excellent behaviour of this new scheme on test cases from the literature. 
Available articles on numerical methods for variational inequalities do not seem
to present any test case with analytic solutions, on which theoretical rates of convergence can be rigorously checked; we therefore develop
in Section \ref{sec:numer} such a test case, and we use it to assess the practical convergence rates of HMM for the Signorini problem.
Section \ref{proof} is devoted to the proof of the main results. Since we deal here with possibly nonconforming schemes, the technique used in \cite{F1} cannot be used to obtain error estimates. Instead, we develop a similar technique as in \cite{B2} for gradient schemes for PDEs.
Dealing with variational inequalities however requires us to establish new estimates with additional terms. These additional terms apparently lead to a
degraded rate of convergence, but we show that the optimal rates can be easily recovered
for many classical methods. We present in Section \ref{sec:approxbarriers} an extension of
our results to the case of approximate barriers, which is natural in many schemes in which the
exact barrier is replaced with an interpolant depending on the approximation space.
The paper is completed by two short sections: a conclusion, and an appendix recalling the
definition of the normal trace of $H_{\div}$ functions.



\section{Assumptions and main results}
\label{main result}


\subsection{Weak formulations} 
Let us start by stating our assumptions and weak formulations for the
Signorini and the obstacle problems.

\begin{hypothesis}[Signorini problem] We make the following assumptions on the data in \eqref{seepage1}--\eqref{sigcondition}:
\begin{enumerate}
\item $\Omega$ is an open bounded connected subset of $\mathbb{R}^{d}$, $d\ge 1$ and $\O$ has a Lipschitz boundary,
\item ${\Lambda}$ is a measurable function from $\Omega$ to $\mathrm{M}_{d}(\mathbb{R})$ (where $\mathrm{M}_{d}(\mathbb{R})$ is the set of $d\times d$ matrices) and there exists $\underline{\lambda}$, $\overline{\lambda} >0$ such that, for a.e. $x \in \Omega$, $\Lambda(x)$ is symmetric with eigenvalues in $[\underline{\lambda},\overline{\lambda}]$,
\item $\Gamma_{1}, \Gamma_{2}$ and $\Gamma_{3}$ are measurable pairwise disjoint
subsets of $\partial\Omega$ such that $\Gamma_{1}\cup \Gamma_{2} \cup \Gamma_{3} = \partial\Omega$ and $\Gamma_{1}$ has a positive $(d-1)$-dimensional measure,
\item $f \in L^{2}(\Omega)$, $a\in L^2(\partial\Omega)$.
\end{enumerate}
\label{hyseepage}\end{hypothesis}

Under Hypothesis \ref{hyseepage}, the weak formulation of Problem \eqref{seepage1}--\eqref{sigcondition} is 
\begin{equation}\label{weseepage}
\left\{
\begin{array}{ll}
\dsp \mbox{Find}\; \bar{u} \in \mathcal{K}:=\{ v \in H^{1}(\Omega)\; : \; \gamma(v)=0\; \mbox{on}\; \Gamma_{1},\; \gamma(v) \leq a\; \mbox{on}\; \Gamma_{3}\}\; \mbox{ s.t.},
\\
\dsp \forall v\in \mathcal{K}\,,\quad
\int_\O \Lambda(x)\nabla \bar{u}(x) \cdot \nabla(\bar{u}-v)(x)\ud x
\leq \int_\O f(x)(\bar{u}(x)-v(x))\ud x\,,
\end{array}
\right.
\end{equation}
where $\gamma:H^1(\O) \mapsto H^{1/2}(\partial\O)$ is the trace operator.
We refer the reader to \cite{A2-17} for the proof of equivalence between the strong and the weak formulations. It was shown in \cite{G10} that, if $\mathcal K$ is not empty (which is the case, for example, if
$a\ge 0$ on $\partial\Omega$) then there exists a unique solution to Problem \eqref{weseepage}. In the sequel, we will assume that the barrier $a$ is such that $\mathcal K$ is not empty.

\begin{hypothesis}[Obstacle problem] Our assumptions on the data in \eqref{obs1}--\eqref{obs4} are:
\begin{enumerate}
\item $\Omega$ and $\Lambda$ satisfy (1) and (2) in Hypothesis \ref{hyseepage},
\item $f \in L^{2}(\O)$, $g \in L^{2}(\O)$.
\end{enumerate}
\label{hyobs}\end{hypothesis}

Under Hypotheses \ref{hyobs}, we consider the obstacle problem \eqref{obs1}--\eqref{obs4}
under the following weak form: 
\begin{equation}\label{weobs}
\left\{
\begin{array}{ll}
\dsp \mbox{Find}\; \bar{u} \in \mathcal{K}:=\{ v \in H_{0}^{1}(\Omega)\; : \; v \leq g\; \mbox{in}\; \Omega\}\;\mbox{ such that, for all $v\in \mathcal{K}$,}\\
\dsp\int_\O \Lambda(x)\nabla \bar{u}(x) \cdot \nabla(\bar{u}-v)(x)\ud x
\leq \int_\O f(x)(\bar{u}(x)-v(x))\ud x.
\end{array}
\right.
\end{equation}
It can be seen \cite{G11} that if $\bar{u}\in C^{2}(\overline{\O})$ and $\Lambda$ is Lipschitz continuous, then 
\eqref{obs1}--\eqref{obs4} and \eqref{weobs} are indeed equivalent. The proof of this equivalence can be easily adapted
to the case where the solution belongs to $H^{2}(\O)$. If $g$ is such that $\mathcal K$ is not empty, which we assume from here on, then \eqref{weobs} has a unique solution



\subsection{Construction of gradient schemes}
\label{gradient discretisation}
Gradient schemes provide a general formulation of different numerical methods. Each gradient scheme is based on a gradient discretisation, which is a set of discrete space
and operators used to discretise the weak formulation of the problem under study. Actually, a gradient
scheme consists of replacing the continuous space and operators used in the weak formulations by
the discrete counterparts provided by a gradient discretisation. In this part, we define gradient discretisations for the Signorini and the obstacle problems, and we list the properties
that are required of a gradient discretisation to give rise to a converging gradient scheme.


\subsubsection{Signorini problem} \label{sec:signorini_problem}

\begin{definition}(Gradient discretisation for the Signorini problem). A gradient discretisation $\mathcal{D}$ for Problem \eqref{weseepage} is defined by $\mathcal{D}=(X_{\mathcal{D}, \Gamma_{1}}, \Pi_{\mathcal{D}}, \mathbb{T}_{\mathcal{D}}, \nabla_{\mathcal{D}})$, where
\begin{enumerate}
\item the set of discrete unknowns $X_{\mathcal{D}, \Gamma_{1}}$ is a finite dimensional vector space on $\mathbb{R}$, taking into account the homogeneous boundary conditions on $\Gamma_{1}$,
\item the linear mapping $\Pi_{\mathcal{D}} : X_{\mathcal{D}, \Gamma_{1}} \longrightarrow L^{2}(\Omega)$ is the reconstructed function,
\item the linear mapping $\mathbb{T}_{\mathcal{D}} : X_{\mathcal{D}, \Gamma_{1}}\longrightarrow H_{\Gamma_{1}}^{1/2}(\partial\Omega)$ is the reconstructed trace,
where $H^{1/2}_{\Gamma_1}(\partial\O)=\{v\in H^{1/2}(\partial\O)\,:\,v=0\mbox{ on }\Gamma_1\}$,
\item the linear mapping $\nabla_{\mathcal{D}} : X_{\mathcal{D}, \Gamma_{1}} \longrightarrow L^{2}(\Omega)^{d}$ is a reconstructed gradient, which must be defined such that $\|\nabla_{\mathcal{D}}\cdot\|_{L^{2}(\Omega)^{d}}$ is a norm on $X_{\mathcal{D},{\Gamma_1}}$.
\end{enumerate}
\label{gdseepage}\end{definition}

As explained above, a gradient scheme for \eqref{weseepage}
is obtained by simply replacing in the weak formulation of the problem
the continuous space and operators by
the discrete space and operators coming from a gradient discretisation.

\begin{definition}(Gradient scheme for the Signorini problem). Let $\mathcal{D}$ be the gradient dicretisation in the sense of Definition \ref{gdseepage}. The gradient scheme scheme for Problem \eqref{weseepage} is
\begin{equation}\label{gsseepage}
\left\{
\begin{array}{ll}
\dsp \mbox{Find}\; u \in \mathcal{K}_{\mathcal{D}}:=\{ v \in X_{\mathcal{D},\Gamma_1}\; : \; \mathbb{T}_{\mathcal{D}}v \leq a\; \mbox{on}\; \Gamma_{3}\}\; \mbox{ such that, $\forall v\in \mathcal{K}_{\mathcal{D}}$},\\
\dsp\int_\O \Lambda(x)\nabla_{\mathcal{D}}u(x) \cdot \nabla_{\mathcal{D}}(u-v)(x)\ud x
\leq \int_\O f(x)\Pi_{\mathcal{D}}{(u-v)(x)}\ud x.
\end{array}
\right.
\end{equation}
\end{definition}

The accuracy of a gradient discretisation is measured through three indicators, related to its
\emph{coercivity}, \emph{consistency} and \emph{limit-conformity}. The good behaviour of these indicators
along a sequence of gradient discretisations ensures that the solutions to the corresponding gradient
schemes converge towards the solution to the continuous problem.

To measure the coercivity of a gradient discretisation $\mathcal{D}$ in the sense of Definition \ref{gdseepage}, we define the norm $C_{\mathcal{D}}$ of the linear mapping $\Pi_{\mathcal D}$ by
\begin{eqnarray} C_{\mathcal{D}} =  {\displaystyle \max_{v \in X_{\mathcal{D},\Gamma_{1}}\setminus\{0\}}\Big(\frac{\|\Pi_{\mathcal{D}}v\|_{L^{2}(\Omega)}}{\|\nabla_{\mathcal{D}} v\|_{L^{2}(\Omega)^{d}}}}+ \frac{\|\mathbb{T}_{\mathcal{D}}v\|_{H^{1/2}(\partial\Omega)}}{\|\nabla_{\mathcal{D}} v\|_{L^{2}(\Omega)^{d}}}\Big).
\label{coercivityseepage}
\end{eqnarray}

The consistency of a gradient discretisation $\mathcal{D}$ in the sense of Definition \ref{gdseepage} is measured by $S_{\mathcal{D}} : \mathcal{K}\times \mathcal K_\disc \longrightarrow [0, +\infty)$ defined by
\begin{equation}
\forall (\varphi,v) \in \mathcal{K}{\times}{\mathcal K_\disc}
, \; S_{\mathcal{D}}(\varphi,v)= \| \Pi_{\mathcal{D}} v - \varphi \|_{L^{2}(\Omega)} 
+ \| \nabla_{\mathcal{D}} v - \nabla \varphi \|_{L^{2}(\Omega)^{d}}.
\label{consistencyseepage}
\end{equation}

To measure the limit-conformity of the gradient discretisation $\mathcal{D}$ in the sense of Definition \ref{gdseepage}, we introduce $W_{\mathcal{D}} : H_{\rm{div}}(\Omega) \longrightarrow [0, +\infty)$ defined by
\begin{multline}
\forall \bpsi \in H_{\rm{div}}(\Omega),\\
W_{\mathcal{D}}(\bpsi)
 = \sup_{v\in X_{\mathcal{D},\Gamma_{1}}\setminus \{0\}}\frac{1}{\|\nabla_{\mathcal{D}} v\|_{L^{2}(\Omega)^{d}}} \Big|\int_{\Omega}(\nabla_{\mathcal{D}}v\cdot \bpsi + \Pi_{\mathcal{D}}v \cdot\mathrm{div} (\bpsi)) \ud x
 -\langle \gamma_{\bfn}(\bpsi), \mathbb{T}_{\mathcal{D}}v \rangle \Big|,
\label{conformityseepage}
\end{multline}
where $H_{\rm{div}}(\Omega)= \{\bpsi \in L^{2}(\Omega)^{d}{\;:\;} \mathrm{div}\bpsi \in L^{2}(\Omega)\}$. We refer the reader to the appendix (Section \ref{traceoperator})
for the definitions of the normal trace $\gamma_{\bfn}$ on $H_{\div}(\O)$, and of
the duality product $\langle \cdot, \cdot\rangle$ between $(H^{1/2}(\partial\Omega))'$ and $H^{1/2}(\partial\Omega)$.

In order for a sequence of gradient discretisations $(\mathcal D_m)_{m\in\NN}$ to provide
converging gradient schemes for PDEs, it is expected that $(C_{\mathcal D_m})_{m\in\NN}$ remains bounded
and that $(\inf_{v\in \mathcal K_{\disc_m}}S_{\mathcal D_m}(\cdot,v))_{m\in\NN}$
and $(W_{\mathcal D_m})_{m\in\NN}$ converge pointwise to $0$. These properties are precisely
called the \emph{coercivity}, the \emph{consistency} and the \emph{limit-conformity} of
the sequence of gradient discretisations $(\mathcal D_m)_{m\in\NN}$, see \cite{B1}.

The definition \eqref{coercivityseepage} of $C_{\mathcal D}$ does not only include
the norm of $\Pi_{\mathcal D}$, as
in the obstacle problem case detailed below, but also quite naturally
the norm of the other reconstruction operator $\mathbb{T}_{\mathcal D}$.

The definition \eqref{conformityseepage} takes into account the non-zero boundary conditions
on a part of $\partial\Omega$. This was already noticed in the case of gradient discretisations adapted for
PDEs with mixed boundary conditions, but a major difference must be raised here.
$W_{\mathcal{D}}$ will be applied to $\bpsi = \Lambda\nabla\bar{u}$. For PDEs \cite{B1,B2,S1}
the boundary conditions ensure that $\gamma_{\bfn}(\bpsi) \in L^{2}(\partial\Omega)$ and thus that $\langle \gamma_{\bfn}(\bpsi), \mathbb{T}_{\mathcal{D}}v \rangle$
can be replaced with $\int_{\Gamma_{3}}\gamma_{\bfn}(\bpsi)\mathbb{T}_{\mathcal{D}}{v} \ud x$
in $W_\disc$. Hence, in the study of gradient schemes for PDEs with mixed boundary conditions, $\mathbb{T}_{\mathcal{D}}v$ only needs to be in $L^{2}(\partial\Omega)$. For the Signorini problem we cannot ensure that $\gamma_{\bfn}(\bpsi) \in L^{2} (\partial\Omega)$;
we only know that $\gamma_{\bfn}(\bpsi) \in H^{1/2}(\partial\Omega)$. 
The definition of $\mathbb{T}_\disc$ therefore needed to be changed
to ensure that this reconstructed trace
takes values in $H^{1/2}(\partial\Omega)$ instead of $L^2(\partial\Omega)$.


\subsubsection{Obstacle problem}

Gradient schemes for the obstacle problems already appear in \cite{Y}. We just recall the
following definitions for the sake of legibility.

The definition of a gradient discretisation for the obstacle problem is not different from
the definition of a gradient discretisation for elliptic PDEs with homogeneous Dirichlet conditions
\cite{B1,B2}. 

\begin{definition}
(Gradient discretisation for the obstacle problem). A gradient discretisation $\mathcal{D}$ for homogeneous Dirichlet boundary conditions is defined by a triplet $\mathcal{D}=(X_{\mathcal{D},0}, \Pi_{\mathcal{D}},\nabla_{\mathcal{D}})$, where
\begin{enumerate}
\item the set of discrete unknowns $X_{\mathcal{D},0}$ is a finite dimensional vector space over $\mathbb{R}$, taking into account the boundary condition \eqref{obs4}.
\item the linear mapping $\Pi_{\mathcal{D}} : X_{\mathcal{D},0} \longrightarrow L^{2}(\Omega)$ gives the reconstructed function,
\item the linear mapping $\nabla_{\mathcal{D}} : X_{\mathcal{D},0} \longrightarrow L^{2}(\Omega)^{d}$ gives a reconstructed gradient, which must be defined such that $\|\nabla_{\mathcal{D}}\cdot\|_{L^{2}(\Omega)^{d}}$ is a norm on $X_{\mathcal{D},0}$.
\end{enumerate}
\label{gdobs}\end{definition}


\begin{definition}
(Gradient scheme for the obstacle problem). Let $\mathcal{D}$ be a gradient discretisation in the sense of Definition \ref{gdobs}. The corresponding gradient scheme for (\ref{weobs}) is given by
\begin{equation}\label{gsobs}
\left\{
\begin{array}{ll}
\dsp \mbox{Find}\; u \in \mathcal{K}_{\mathcal{D}}:=\{ v \in X_{\mathcal{D},0}\; : \; \Pi_{\mathcal{D}} v \leq g\; \mbox{in}\; \Omega\}\; 
\mbox{ such that, $\forall v\in \mathcal{K}_{\mathcal{D}}$},\\
\dsp\int_\O \Lambda(x)\nabla_{\mathcal{D}}u(x) \cdot \nabla_{\mathcal{D}}(u-v)(x)\ud x
\leq \int_\O f(x)\Pi_{\mathcal{D}}{(u-v)(x)}\ud x.
\end{array}
\right.
\end{equation}
\end{definition}

The coercivity, consistency and limit-conformity of a gradient discretisation in the sense
of Definition \ref{gdobs} are defined through the following constant and functions:

\begin{eqnarray} C_{\mathcal{D}} =  {\displaystyle \max_{v \in X_{\mathcal{D},0}\setminus\{0\}}\frac{\|\Pi_{\mathcal{D}}v\|_{L^{2}(\Omega)}}{\|\nabla_{\mathcal{D}} v\|_{L^{2}(\Omega)^{d}}}},
\label{coercivityobs}
\end{eqnarray}
\begin{equation}
\forall (\varphi,v) \in \mathcal{K}\times \mathcal K_\disc, \; S_{\mathcal{D}}(\varphi,v)= \| \Pi_{\mathcal{D}} v - \varphi \|_{L^{2}(\Omega)} 
+ \| \nabla_{\mathcal{D}} v - \nabla \varphi \|_{L^{2}(\Omega)^{d}},
\label{consistencyobs}
\end{equation}
and
\begin{equation}
\forall \bpsi \in H_{\mathrm{div}}(\Omega),\;
W_{\mathcal{D}}(\bpsi)=
 \sup_{v\in X_{\mathcal{D},0}\setminus \{0\}}\frac{1}{\|\nabla_{\mathcal{D}} v\|_{L^{2}(\Omega)^{d}}} \Big|\int_{\Omega}(\nabla_{\mathcal{D}}v\cdot \bpsi + \varPi_{\mathcal{D}}v \cdot\mathrm{div} (\bpsi)) \ud x
\Big|.
\label{conformityobs}
\end{equation}

The only indicator that changes with respect to \cite{B1,B2} is $S_\disc$. For PDEs,
the consistency requires to consider $S_{\mathcal D}$ on $H^1_0(\Omega)\times X_{\mathcal D,0}$
and to ensure that $\inf_{v\in X_{\mathcal D_m}}S_{\mathcal D_m}(\varphi,v)\to 0$
as $m\to\infty$ for any $\varphi\in H^1_0(\Omega)$.
Here, the domain of $\mathcal S_{\mathcal D}$ is adjusted to the set to which
the solution of the variational inequality belongs (namely $\mathcal K$), and to the set in which
we can pick the test functions of the gradient scheme (namely $\mathcal K_\disc$).

We note that this double reduction does not necessarily facilitate, with respect to the PDEs
case, the proof of the consistency of the sequence of gradient discretisations. In practice, however,
the proof developed for the PDE case also provides the consistency of gradient discretisations
for variational inequalities.




\subsection{Error estimates}
We present here our main error estimates for the gradient schemes approximations of Problems \eqref{weseepage} and \eqref{weobs}.


\subsubsection{Signorini problem}

\begin{theorem}[Error estimate for the Signorini problem] Under Hypothesis \ref{hyseepage},
let $\bar{u} \in \mathcal{K}$ be the solution to Problem (\ref{weseepage}). 
If $\mathcal{D}$ is a gradient discretisation in the sense of Definition \ref{gdseepage} and $\mathcal{K}_{\mathcal{D}}$ is non-empty, then there exists a unique solution $u \in \mathcal K_\disc$ to the gradient scheme (\ref{gsseepage}). Furthermore, this solution satisfies the following inequalities, for any $v_\disc\in\mathcal K_\disc$:
\begin{multline}
\|\nabla_{\mathcal{D}}u - \nabla\bar{u}\|_{L^{2}(\Omega)^{d}} \leq\\
\sqrt{\frac{2}{\underline{\lambda}}G_{\mathcal{D}}(\bar{u},v_{\mathcal{D}})^{+}}
+
\frac{1}{\underline{\lambda}}
[W_{\mathcal{D}}(\Lambda\nabla\bar{u})+(\overline{\lambda}+\underline{\lambda})
S_{\mathcal{D}}(\bar{u}, v_{\mathcal{D}})]
,
\label{errorseepage1}
\end{multline}
and
\begin{multline}
\|\Pi_{\mathcal{D}}u - \bar{u}\|_{L^{2}(\Omega)} \leq\\
C_{\mathcal{D}}\sqrt{\frac{2}{\underline{\lambda}}G_{\mathcal{D}}(\bar{u},v_{\mathcal{D}})^+}+
\frac{1}{\underline{\lambda}}
[{C_\disc}W_{\mathcal{D}}(\Lambda\nabla\bar{u})+
(C_{\mathcal{D}}\overline{\lambda}+\underline{\lambda})S_{\mathcal{D}}(\bar{u},v_{\mathcal{D}})],
\label{errorseepage2}\end{multline}
where
\begin{equation}
G_{\mathcal{D}}(\bar{u}, v_{\mathcal{D}})= \langle \gamma_\bfn(\Lambda\nabla\bar u), \mathbb{T}_{\mathcal{D}}v_{\mathcal{D}}-\gamma\bar{u} \rangle 
\quad\mbox{ and }\quad G_{\disc}^{+}= {\rm max}(0,G_{\disc}).  
\end{equation}
Here, the quantities $C_{\mathcal{D}}$, $S_{\mathcal{D}}(\bar{u}$, $v_{\mathcal{D}})$ and $W_{\mathcal{D}}$ are defined by \eqref{coercivityseepage}, \eqref{consistencyseepage} and \eqref{conformityseepage}. 
\label{errorseepage}
\end{theorem}


\subsubsection{Obstacle problem}

An error estimate for gradient schemes for the obstacle problem is already given
in \cite[Theorem 1]{Y}. For low-order methods with piecewise constant
approximations, such as the HMM schemes, this theorem provides an $\mathcal O(\sqrt{h})$ rate
of convergence for the function and the gradient
($h$ is the mesh size). The following theorem improves \cite[Theorem 1]{Y}
by introducing the free choice of interpolant $v_\disc$. This enables us, in Section \ref{improveobs}, to
establish much better rates of convergence --
namely $\mathcal O(h)$ for HMM, for example.

\begin{theorem}[Error estimate for the obstacle problem] Under Hypothesis \ref{hyobs}, let $\bar{u} \in \mathcal{K}$ be the solution to Problem (\ref{weobs}). If $\mathcal{D}$ is a gradient discretisation in the sense of Definition \ref{gdobs} and if $\mathcal{K}_{\mathcal{D}}$ is non-empty, then there exists a unique solution $u \in \mathcal{K}_{\mathcal{D}}$ to the gradient scheme (\ref{gsobs}). Moreover, if $\mathrm{div}(\Lambda\nabla\bar{u})\in L^{2}(\Omega)$
then this solution satisfies the following estimates, for any $v_\disc\in \mathcal K_\disc$:
\begin{multline}
\|\nabla_{\mathcal{D}}u - \nabla\bar{u}\|_{L^{2}(\Omega)^{d}} \leq\\
\sqrt{\frac{2}{\underline{\lambda}}E_{\mathcal{D}}(\bar{u},v_{\mathcal{D}})^{+}}
+
\frac{1}{\underline{\lambda}}
[W_{\mathcal{D}}(\Lambda\nabla\bar{u})+(\overline{\lambda}+\underline{\lambda})
S_{\mathcal{D}}(\bar{u}, v_{\mathcal{D}})]
,
\label{errorobs1}
\end{multline}
\begin{multline}
\|\Pi_{\mathcal{D}}u - \bar{u}\|_{L^{2}(\Omega)} \leq\\
C_{\mathcal{D}}\sqrt{\frac{2}{\underline{\lambda}}E_{\mathcal{D}}(\bar{u},v_{\mathcal{D}})^+}+
\frac{1}{\underline{\lambda}}
[{C_\disc}W_{\mathcal{D}}(\Lambda\nabla\bar{u})+
(C_{\mathcal{D}}\overline{\lambda}+\underline{\lambda})S_{\mathcal{D}}(\bar{u},v_{\mathcal{D}})],
\label{errorobs2}\end{multline}
where 
\begin{equation}
E_{\mathcal{D}}(\bar
u,v_{\mathcal{D}})=\int_{\Omega}(\mathrm{div}(\Lambda\nabla\bar{u})+f)(\bar{u}-\Pi_{\mathcal{D}}v_{\mathcal{D}})\ud x \quad
\mbox{ and } \quad E_{\disc}^{+}= {\rm max}(0,E_{\disc}). 
\end{equation}
Here, the quantities $C_{\mathcal{D}}$, $S_{\mathcal{D}}(\bar{u}, v_{\mathcal{D}})$ and $W_{\mathcal{D}}$ are defined by \eqref{coercivityobs}, \eqref{consistencyobs} and \eqref{conformityobs}. 
\label{errorobs}\end{theorem}

\begin{remark} We note that, if $\Lambda$ is Lipschitz-continuous, the assumption $\div(\Lambda\nabla\bar{u}) \in L^{2}(\Omega)$ is reasonable given the $H^{2}$-regularity result on $\bar u$
of \cite{A2-23-4}.
\end{remark}

Compared with the previous general estimates, such as Falk's Theorem \cite{F1}, our estimates seem to be 
simpler since it only involves one choice of interpolant  $v_\disc$ while Falk's estimate depends on  
choices of $v_{h} \in K_{h}$ and $v \in K$. Yet, our estimates
provide the same final orders of convergence as Falk's estimate.
We also note that Estimates \eqref{errorseepage1}, \eqref{errorseepage2}, \eqref{errorobs1} and \eqref{errorobs2} are applicable to conforming and non-conforming methods, whereas
the estimates in \cite{F1} seem to be applicable only to conforming methods.


\subsection{Orders of convergence}

In what follows, we consider polytopal open sets $\Omega$ and gradient discretisations based on
polytopal meshes of $\Omega$, as defined in \cite{B10,S1}.
The following definition is a simplification (it does not include
the vertices) of the definition in \cite{B10}, that will however
be sufficient to our purpose.

\begin{definition}[Polytopal mesh]\label{def:polymesh}~
Let $\Omega$ be a bounded polytopal open subset of $\RR^d$ ($d\ge 1$). 
A polytopal mesh of $\O$ is given by $\polyd = (\mesh,\edges,\centers)$, where:
\begin{enumerate}
\item $\mesh$ is a finite family of non empty connected polytopal open disjoint subsets of $\O$ (the cells) such that $\overline{\O}= \dsp{\cup_{K \in \mesh} \overline{K}}$.
For any $K\in\mesh$, $|K|>0$ is the measure of $K$ and $h_K$ denotes the diameter of $K$.

\item $\edges$ is a finite family of disjoint subsets of $\overline{\O}$ (the edges of the mesh in 2D,
the faces in 3D), such that any $\edge\in\edges$ is a non empty open subset of a hyperplane of $\RR^d$ and $\edge\subset \overline{\O}$.
We assume that for all $K \in \mesh$ there exists  a subset $\edgescv$ of $\edges$
such that $\dr K  = \dsp{\cup_{\edge \in \edgescv}} \overline{\edge}$. 
We then denote by $\mesh_\edge = \{K\in\mesh\,:\,\edge\in\edgescv\}$.
We then assume that, for all $\edge\in\edges$, $\mesh_\edge$ has exactly one element
and $\edge\subset\partial\O$, or $\mesh_\edge$ has two elements and
$\edge\subset\O$. 
We let $\edgesint$ be the set of all interior faces, i.e. $\edge\in\edges$ such that $\edge\subset \O$, and $\edgesext$ the set of boundary
faces, i.e. $\edge\in\edges$ such that $\edge\subset \dr\O$.
For $\edge\in\edges$, the $(d-1)$-dimensional measure of $\edge$ is $|\edge|$,
the centre of gravity of $\edge$ is $\centeredge$, and the diameter of $\edge$ is $h_\edge$.

\item $\centers = (x_K)_{K \in \mesh}$ is a family of points of $\O$ indexed by $\mesh$ and such that, for all  $K\in\mesh$,  $\xcv\in K$ ($\xcv$ is sometimes called the ``centre'' of $\cv$). 
We then assume that all cells $K\in\mesh$ are  strictly $\xcv$-star-shaped, meaning that 
if $x$ is in the closure of $K$ then the line segment $[\xcv,x)$ is included in $K$.

\end{enumerate}
For a given $K\in \mesh$, let $\bfn_{K,\sigma}$ be the unit vector normal to $\sigma$ outward to $K$
and denote by $d_{K,\sigma}$ the orthogonal distance between $x_K$ and $\sigma\in\mathcal E_K$.
The size of the discretisation is $h_\mesh=\sup\{h_K\,:\; K\in \mesh\}$.
\end{definition}

For most gradient discretisations for PDEs based on first order methods, including HMM methods and conforming or non-conforming finite elements methods, explicit estimates on $S_\disc$ and $W_\disc$ can be established \cite{S1}. The proofs of these estimates are easily transferable to the above setting of gradient discretisations for variational inequalities, and give
\begin{equation}
W_{\mathcal{D}}(\bpsi) \leq Ch_{\mathcal M} \| \bpsi \|_{H^{1}(\Omega)^{d}}, \quad \forall \bpsi \in H^{1}(\Omega)^{d},
\end{equation}
\begin{equation}
{\rm{inf}}_{v_{\mathcal{D}}\in \mathcal{K}_{\mathcal{D}}}S_{\mathcal{D}}(\varphi,v_{\mathcal{D}}) \leq Ch_{\mathcal M} \| \varphi \|_{H^{2}(\Omega)}, \quad \forall \varphi \in H^{2}(\Omega) \cap \mathcal K.
\end{equation}
Hence, if $\Lambda$ is Lipschitz-continuous and $\bar u\in H^2(\O)$, then the terms involving $W_\disc$ and $S_\disc$ in \eqref{errorseepage1}, \eqref{errorseepage2}, \eqref{errorobs1} 
and \eqref{errorobs2} are of order
$\mathcal O(h_{\mathcal M})$, provided that $v_{\mathcal D}$ is chosen to optimise $S_{\disc}$.
Since the terms $G_{\mathcal D}$ and $E_{\mathcal D}$ can be bounded
above by $C||\gamma(\bar u)-
\mathbb{T}_\disc v_\disc||_{L^2(\partial\O)}$ and $C||\bar u-\Pi_\disc v_\disc||_{L^2(\O)}$, respectively, we expect these to be
of order $h_{\mathcal M}$. Hence, the dominating terms in the error estimates are
$\sqrt{G_\disc^+}$ and $\sqrt{E_\disc^+}$. In first approximation, these terms seem to behave
as $\sqrt{h_{\mathcal M}}$ for a first order conforming or non-conforming method. This initial brute estimate
can however usually be improved, as we will show in Theorems \ref{theorem:improveseepage} and
\ref{theorem:improveobse}, even for non-conforming methods based on piecewise constant
reconstructed functions, and lead to the expected $\mathcal O(h_{\mathcal M})$ global convergence rate.


\subsubsection{Signorini problem} 
\label{sec:improvesig}

For a first order conforming numerical method, the $\mathbb{P}_1$ finite element method for example, 
if $\bar u$ is in $H^2(\O)$, then the classical interpolant $v_\disc\in X_{\disc,\Gamma_1}$ constructed from the values of $\bar u$ at the vertices satisfies \cite{C3}
\begin{align*}
\| \Pi_{\mathcal{D}}v_{\mathcal{D}} - \bar{u} \|_{L^{2}(\Omega)} &\leq C h_{\mathcal M}^{2} \| D^{2}\bar{u}\|_{L^{2}(\Omega)^{d\times d}},\\
\| \nabla_{\mathcal{D}}v_{\mathcal{D}} - \nabla\bar{u} \|_{L^{2}(\Omega)^d} &\leq Ch_{\mathcal M} \| D^{2}\bar{u}\|_{L^{2}(\Omega)^{d\times d}},\\
||\mathbb{T}_\disc v_\disc-\gamma(\bar u)
||_{L^2(\partial\O)}&\leq Ch_{\mathcal M}^2 \| D^{2}\gamma(\bar{u})\|_{L^{2}(\partial\Omega)^{d\times d}}.
\end{align*}
If $a$ is piecewise affine on the mesh then this interpolant $v_\disc$ lies in $\mathcal K_\disc$ and can therefore be used
in Theorem \ref{errorseepage}. We then have $S_{\mathcal{D}}(\bar{u},v_{\mathcal{D}})= \mathcal{O}(h_\mesh)$ and $G_{\mathcal{D}}(\bar{u},v_{\mathcal{D}})= \mathcal{O}(h_\mesh^{2})$, and \eqref{errorseepage1}--\eqref{errorseepage2} give an order one error estimate on
the $H^1$ norm. If $a$ is not linear, the definition of $\mathcal K_\disc$ is usually relaxed, see Section \ref{sec:approxbarriers}.

A number of low-order methods have piecewise constant approximations of the solution,
e.g. finite volume methods or finite element methods with mass lumping. For those,
there is no hope of finding an interpolant which gives an order $h_{\mathcal M}^2$ approximation
of $\bar u$ in $L^2(\O)$ norm. We can however prove, for such methods, that $G_\disc$ behaves better than the expected $\sqrt{h_\mesh}$ order.
In the following theorem, we denote by $W^{2,\infty}(\partial\O)$ the functions
$v$ on $\partial\O$ such that, for any face $F$ of $\partial\O$, $v_{|F}\in W^{2,\infty}(F)$.

\begin{theorem}[Signorini problem: order of convergence for non-conforming reconstructions]
Under Hypothesis \ref{hyseepage}, let $\disc$ be a gradient discretisation in the sense of
Definition \ref{gdseepage}, such that $\mathcal K_\disc\not=\emptyset$, and let $\polyd=(\mesh,\edges,\centers)$ be a polytopal mesh of $\O$.
Let $\bar{u}$ and $u$ be the respective solutions to Problems \eqref{weseepage} and \eqref{gsseepage}.
We assume that the barrier $a$ is constant, that $\gamma_\bfn(\Lambda\nabla\bar u)\in L^2(\partial\O)$ and that $\gamma(\bar u)\in W^{2,\infty}(\partial\O)$.
We also assume that there exists an
interpolant $v_{\mathcal{D}} \in \mathcal{K}_{\mathcal{D}}$ such that $S_{\mathcal{D}}(\bar{u},v_{\mathcal{D}})\le C_1h_{\mathcal M}$
and that, for any $\sigma\in\edgesext$, $||\mathbb{T}_{\mathcal{D}}v_{\mathcal{D}}-\gamma(\bar{u})(x_{\sigma})||_{L^2(\edge)}\le C_2 h_\sigma^2|\edge|^{1/2}$, where $x_{\sigma}\in\sigma$ (here, the constants $C_i$ do not depend on the edge or the mesh). Then, there exists $C$ depending only on 
$\O$, $\Lambda$, $\bar u$, $a$, $C_1$, $C_2$ and an upper bound of $C_\disc$ such that
\begin{equation}
\|\nabla_{\mathcal{D}}u - \nabla\bar{u}\|_{L^{2}(\Omega)^{d}}
+ \|\Pi_{\mathcal{D}}u - \bar{u}\|_{L^{2}(\Omega)}
\leq Ch_{\mathcal M}+CW_{\disc}(\Lambda\nabla\bar u).
\label{eq:improverrorseepage}\end{equation}
\label{theorem:improveseepage}
\end{theorem}

\begin{remark}\label{rem:makesense}
Since $\gamma_\bfn(\Lambda\nabla\bar u) \in L^2(\partial\O)$ the reconstructed
trace $\mathbb{T}_\disc$ can be taken with values in $L^2(\partial\Omega)$
(see the discussion at the end of Section \ref{sec:signorini_problem}).
In particular, $\mathbb{T}_\disc v$ can be
piecewise constant. The $H^{1/2}(\partial\Omega)$ norm in the definition
\eqref{coercivityseepage} of $C_\disc$
is then replaced with the $L^2(\partial\O)$ norm, and the duality product in \eqref{conformityseepage}
is replaced with a plain integral on $\partial\O$.
\end{remark}

For the two-point finite volume method,
an error estimate similar to \eqref{eq:improverrorseepage} is stated in \cite{A2}, under the assumption that the solution is in $H^2(\O)$. It however seems
to us that the proof in \cite{A2} uses an estimate which requires $\gamma(\bar u)\in W^{2,\infty}(\dr\O)$, as in Theorem \ref{theorem:improveseepage}.

We notice that the assumption that $a$ is constant is compatible with some models, such as an electrochemical reaction and friction mechanics
\cite{A2-17}. This condition on $a$ can
be relaxed or, as detailed in Section \ref{sec:approxbarriers}, $a$ can be approximated by a simpler barrier -- the definition of which depends on the considered scheme.


\subsubsection{Obstacle problem}
\label{improveobs}

As explained in the introduction of this section,
to estimate the order of convergence for the obstacle problem we only need to estimate $E_{\mathcal{D}}(\bar{u},v_{\mathcal{D}})$.
This can be readily done for $\mathbb{P}_1$ finite elements, for example. If $g$ is linear or constant, letting $v_{\disc} = \bar u$ at all vertices  (as in \cite{finite-element}) shows that $v_{\disc}$ is an element of $\mathcal K_{\disc}$, and that $S_{\disc}(\bar u, v_{\disc})= \mathcal O (h_\mesh)$ and $E_{\disc}(\bar u, v_{\disc})= \mathcal O (h_\mesh^{2})$. Therefore, Theorem \ref{errorobs} provides an order one error estimate.

The following theorem shows that, as for the Signorini problem, the error estimate for
the obstacle problem is of order one even for methods with piecewise constant
reconstructed functions.

\begin{theorem}[Obstacle problem: order of convergence for non-conforming reconstructions]
Under Hypothesis \ref{hyobs}, let $\disc$ be a gradient discretisation in the sense of Definition \ref{gdobs}, such
that $\mathcal K_\disc\not=\emptyset$, and let $\polyd=(\mesh,\edges,\centers)$ be a polytopal mesh of $\O$.
Let $\bar{u}$ and $u$ be the respective solutions to Problems \eqref{weobs} and \eqref{gsobs}.
We assume that $g$ is constant, that $\div(\Lambda\nabla\bar u)\in L^2(\O)$, and that $\bar u\in W^{2,\infty}(\O)$.
We also assume that there exists an interpolant $v_{\mathcal{D}} \in \mathcal{K}_{\mathcal{D}}$ such that $S_{\mathcal{D}}(\bar{u},v_{\mathcal{D}})\le C_1h_{\mathcal M}$ and, for any $K\in\mesh$,
$||\Pi_{\mathcal{D}}v_{\mathcal{D}}-\bar{u}(x_{K})||_{L^2(K)}\le C_2h_K^2|K|^{1/2}$
(here, the various constants $C_i$ do not depend on the cell or the mesh). Then, there exists $C$ depending only on $\O$, $\underline{\lambda}$, $\overline{\lambda}$, $C_1$, $C_2$ and an upper bound of $C_\disc$ such that
\begin{equation}
\|\nabla_{\mathcal{D}}u - \nabla\bar{u}\|_{L^{2}(\Omega)^{d}}
+ \|\Pi_{\mathcal{D}}u - \bar{u}\|_{L^{2}(\Omega)}
\leq Ch_{\mathcal M}{+CW_{\disc}(\Lambda\nabla\bar u).}
\label{eq:improverrorobs}\end{equation}
\label{theorem:improveobse}
\end{theorem}

As for the Signorini problem, under regularity assumptions the condition that $g$ is constant can be relaxed,
see Remark \ref{relax:gconst}. However, if $g$ is not constant it might be more practical to
use an approximation of this barrier in the numerical discretisation -- see Section \ref{sec:approxbarriers}.

\begin{remark}[Convergence without regularity assumptions]
The various regularity assumptions made on the data or the solutions in all previous theorems
are used to state optimal
orders of convergence. The convergence of gradient schemes can however
be established under the sole regularity assumptions stated in Hypothesis \ref{hyseepage} and \ref{hyobs}, see \cite{Y}.
\end{remark}



\section{Examples of gradient schemes}\label{sec:examples}

Many numerical schemes can be seen as gradient schemes. We show here
that two families of methods, namely conforming Galerkin methods and HMM methods,
can be recast as gradient schemes when applied to variational inequalities.


\subsection{Galerkin methods}

Any Galerkin method, including conforming finite elements and spectral methods of any order, fits into the gradient scheme framework.

For the Signorini problem, if $V$ is a finite dimensional subspace of $\{v\in H^1(\O)\,:\,\gamma(v)=0\mbox{ on $\Gamma_1$}\}$, we let 
$X_{\disc,\Gamma_1}=V$, $\Pi_\disc={\rm Id}$, $\nabla_\disc=\nabla$ and $\mathbb{T}_\disc=\gamma$. Then
$(X_{\disc,\Gamma_{1}},\Pi_\disc,\mathbb{T}_\disc,\nabla_\disc)$ is a gradient discretisation and the corresponding gradient scheme
corresponds to the Galerkin approximation of \eqref{weobs}. Then, $C_\disc$ defined by
\eqref{coercivityseepage} is bounded above by the maximum between the Poincar\'e constant
and the norm of the trace operator $\gamma:H^1(\O)\to H^{1/2}(\dr\O)$.

Conforming Galerkin methods for the obstacle problem consist in taking
$X_{\disc,0}=V$
a finite dimensional subspace of $H_0^1(\O)$,
$\Pi_\disc={\rm Id}$ and $\nabla_\disc=\nabla$. For this gradient discretisation,
$C_\disc$ (defined by \eqref{coercivityobs}) is bounded above by the Poincar\'e constant in $H^1_0(\O)$.

For both the Signorini and the obstacle problems, $W_\disc$ is identically zero and the error estimate
is solely dictated by the interpolation error $S_\disc(\bar u,v)$ and by $G_\disc(\bar u,v)$
or $E_\disc(\bar u,v)$, as expected.
For $\mathbb{P}_1$ finite elements, if the barrier ($a$ or $g$) is piecewise affine on the mesh then the classical
interpolant $v$ constructed from the nodal
values of $\bar u$ belongs to $\mathcal K$.
As we saw, using this interpolant in Estimates \eqref{errorobs1} and \eqref{errorobs2} leads to the
expected order 1 convergence.

If the barrier is more complex, then it is usual to consider some piecewise affine approximation of it
in the definition of the scheme, and the gradient scheme is then modified to take into account
this approximate barrier (see Section \ref{sec:approxbarriers}); the previous natural interpolant
of $\bar u$ then belongs to the (approximate) discrete set $\widetilde{\mathcal K}_\disc$
used to define the scheme, and gives
a proper estimate of $S_\disc(\bar u,v)$ and $G_\disc(\bar u,v)$ or $E_\disc(\bar u,v)$.


\subsection{Hybrid mimetic mixed methods (HMM)}\label{sec:hmm}
HMM is a family of methods, introduced in  \cite{B5}, that contains the hybrid finite volume methods \cite{sushi},
the (mixed-hybrid) mimetic finite differences methods \cite{bre-05-fam} and the mixed finite
volume methods \cite{dro-06-mix}. These methods were developed for anisotropic heterogeneous diffusion equations
on generic grids, as often encountered in engineering problems related to flows in porous media.
So far, HMM schemes have only be considered for PDEs.
The generic setting of gradient schemes enable us to develop and analyse HMM methods for variational
inequalities.

\begin{remark} The (conforming) mimetic methods studied in \cite{A-32} for obstacle problems
are different from the mixed-hybrid mimetic method of the HMM family; they are based on nodal unknowns rather than
cell and face unknowns. These conforming mimetic methods however also fit into the gradient
scheme framework \cite{B10}.  \end{remark}


\subsubsection{Signorini problem}\label{sec:HMM.sig}

Let $\polyd$ be a polytopal mesh of $\Omega$.
The discrete space to consider is
\begin{equation}\label{def.XDG}
\begin{aligned}
X_{\disc,\Gamma_{1}}=\{ &v=((v_{K})_{K\in \mathcal{M}}, (v_{\sigma})_{\sigma \in \mathcal{E}})\;:\; v_{K} \in \RR, v_{\sigma} \in \RR,\\
& v_{\sigma}=0 \; \mbox{for all}\; \sigma \in \edgesext\mbox{ such that }\sigma\subset \Gamma_{1} \}
\end{aligned}
\end{equation}
(we assume here that the mesh is compatible with $(\Gamma_i)_{i=1,2,3}$, in the sense that for any $i=1,2,3$,
each boundary edge is either fully included in $\Gamma_i$ or disjoint from this set).
The operators $\Pi_\disc$ and $\nabla_\disc$ are defined as follows:
\[
\forall v\in X_{\disc,\Gamma_1}\,,\;\forall K\in\mathcal M\,:\,\Pi_\disc v=v_K\mbox{ on $K$},
\]
\[
\forall v\in X_{\disc,\Gamma_1}\,,\;\forall K\in\mathcal M\,:\,\nabla_\disc v=\nabla_{K}v+
\frac{\sqrt{d}}{d_{K,\sigma}}(A_{K}R_{K}(v))_{\sigma}\mathbf{n}_{K,\sigma}\mbox{ on }{\rm co}(\{\sigma,x_K\})
\]
where ${\rm co}(S)$ is the convex hull of the set $S$ and
\begin{itemize}
\item $\nabla_{K}v= \frac{1}{|K|}\sum_{\sigma\in \edgescv}|\sigma|(v_{\sigma}-v_K)\mathbf{n}_{K,\sigma}$,
\item $R_{K}(v)=(v_{\sigma}-v_{K}-\nabla_{K}v\cdot (x_{\sigma}-x_{K}))_{\sigma\in\mathcal E_K}\in \RR^{{\rm Card}(\edgescv)}$,
\item $A_K$ is an isomorphism of the vector space ${\rm Im}(R_K)$.
\end{itemize}
It is natural to consider a piecewise constant trace reconstruction:
\begin{equation}\label{def:TD}
\forall \sigma\in\mathcal E_{\rm ext}\,:\,\mathbb{T}_\disc v = v_\sigma\mbox{ on $\sigma$}.
\end{equation}
This reconstructed trace operator does not take values in $H^{1/2}_{\Gamma_1}(\partial\O)$,
and the corresponding gradient discretisation $\disc$ is therefore not admissible
in the sense of Definition \ref{gdseepage}. However, for sufficiently regular $\bar u$,
we can consider reconstructed traces in $L^2(\partial\O)$ only, see Remark \ref{rem:makesense}.

It is proved in \cite{B1} that, if $\Gamma_1=\dr\O$, then $(X_{\disc,\Gamma_1},\Pi_\disc,\nabla_\disc)$
is a gradient discretisation for homogeneous Dirichlet
boundary conditions, whose gradient scheme gives the HMM method when applied to elliptic PDEs.

With $ \mathcal{K}_{\disc}:= \{ v \in X_{\disc,\Gamma_{1}} \; : \; v_{\sigma} \leq a
\mbox{ on $\sigma$, for all $\sigma \in \mathcal{E}_{\rm ext}$ such that $\sigma\subset \Gamma_{3}$}\}$,
the gradient scheme \eqref{gsseepage} corresponding to this gradient discretisation
can be recast as
\begin{equation}\label{HMM-ineq-sig}
\left\{
\begin{array}{l}
\dsp \mbox{Find}\; u \in \mathcal K_\disc \mbox{ such that, for all } v\in \mathcal{K}_{\disc}\,,\\
\dsp\sum_{K\in \mesh}|K|\Lambda_K\nabla_{K}u\cdot \nabla_{K}(u-v)+ \dsp\sum_{K\in \mesh}R_{K}(u-v)^{T}\mathbb{B}_{K}R_{K}(u) \\
~\hfill\leq 
\dsp\sum_{K\in \mesh}(u_{K}-v_{K})\dsp\int_{K}f(x) \ud x.
\end{array}
\right.
\end{equation}
where $\Lambda_K$ is the value of $\Lambda$ in the cell $K$ (it is usual -- although not mandatory -- to assume that $\Lambda$
is piecewise constant on the mesh) and, for $K\in\mesh$,
$\mathbb{B}_K$ is a symmetric positive definite matrix of size ${\rm Card}(\mathcal E_K)$.
This matrix is associated to $A_K$, see \cite{B1} for the details.

Under classical assumptions on the mesh and on the
matrices $\mathbb{B}_K$, it is shown in \cite{B1} that the constant $C_\disc$ can be bounded above by quantities
not depending on the mesh size, and that $W_\disc(\bpsi)$ is of order $\mathcal O(h_\mesh)$.
If $d\le 3$ and $\bar u\in H^2(\O)$, we can construct the interpolant $v_\disc=((v_K)_{K\in\mesh},(v_\edge)_{\edge\in\mesh})$ with $v_K=\bar u(x_K)$ and $v_\edge=\bar u(\centeredge)$
and, by \cite[Propositions 13.14 and 13.15]{S1}, we have 
\[
S_\disc(\bar u,v_\disc)=||\Pi_\disc v_\disc-\bar u||_{L^2(\O)}+ ||\nabla_\disc v_\disc-\nabla u||_{L^2(\O)^d}\le Ch_\mesh||\bar u||_{H^2(\O)}.
\]
Under the assumption that $\bar u\in W^{2,\infty}(\O)$, this estimate is
also proved in \cite{B10} (with $||\bar u||_{H^2(\O)}$ replaced with $||\bar u||_{W^{2,\infty}(\O)}$).
If $\bar u\in\mathcal K$ and $a$ is constant, this interpolant $v_\disc$ belongs to
$\mathcal K_\disc$. Moreover, given the definition \eqref{def:TD} of
$\mathbb{T}_\disc$, we have $(\mathbb{T}_\disc v_\disc)_{|\sigma}-\gamma(\bar u)(\overline{x}_\sigma)=v_\sigma-\bar u(\overline{x}_\sigma)=0$. Hence, using this $v_\disc$ in Theorem \ref{theorem:improveseepage},
we see that
the HMM scheme for the Signorini problem enjoys an order 1 rate of convergence.
A non-constant barrier $a$ can be approximated by a piecewise constant barrier on $\edgesext$
and Theorem \ref{th:improvedapproxseepage} shows that, with this approximation, we still
have an order 1 rate of convergence.
Note that since this convergence also involves the gradients, a first order convergence
is optimal for a low-order method such as the HMM method.


\subsubsection{Obstacle problem}

We still take $\polyd$ a potytopal mesh of $\O$, and we consider $\disc=(X_{\disc,0},\Pi_\disc,\nabla_\disc)$ given by $X_{\disc,0}=X_{\disc,\dr\O}$ (see \eqref{def.XDG})
and $\Pi_\disc$ and $\nabla_\disc$ as in Section \ref{sec:HMM.sig}.
Using $\disc$ in \eqref{gsobs}, we obtain the HMM method for the obstacle problem.
Setting $\mathcal{K}_{\disc}:= \{ v \in X_{\disc,0} \; : \; v_{K} \leq g \; \mbox{on $K$,
for all } K \in \mesh\}$,
this method reads
\begin{equation}\label{HMM-ineq-obs}
\left\{
\begin{array}{l}
\dsp \mbox{Find}\; u \in \mathcal{K}_{\disc}\mbox{ such that, for all } v\in \mathcal{K}_{\disc}\,,\\
\dsp\sum_{K\in \mesh}|K|\Lambda_K\nabla_{K}u\cdot \nabla_{K}(u-v)+ \dsp\sum_{K\in \mesh}R_{K}(u-v)^{T}\mathbb{B}_{K}R_{K}(u) \\
~\hfill\leq 
\dsp\sum_{K\in \mesh}(u_{K}-v_{K})\dsp\int_{K}f(x) \ud x
\end{array}
\right.
\end{equation}

Using the same interpolant as in Section \ref{sec:HMM.sig}, we see that
Theorem \ref{theorem:improveobse} (for a constant barrier $g$) and
Theorem \ref{th:improvedapproxobs} (for a piecewise constant approximation of $g$)
provide an order 1 convergence rate.

To our knowledge, the HMM method has never been considered before for the Signorini or the obstacle problem. These methods are however fully relevant for these models,
since HMM schemes have been developed to deal with anisotropic heterogeneous
diffusion processes in porous media (processes involved in seepage problems for example).
The error estimates obtained here on the HMM method,
through the development of a gradient scheme framework for variational inequalities, are therefore
new and demonstrate that this framework has applications beyond merely giving back known
results for methods such as $\mathbb{P}_1$ finite elements.



\section{Numerical results}\label{sec:numer}

We present numerical experiments to highlight the efficiency of the HMM methods for variational inequalities, and to verify our theoretical results.
To solve the inequalities \eqref{HMM-ineq-obs} and \eqref{HMM-ineq-sig}, we use the monotony algorithm of \cite{A2-23}.
We introduce the fluxes $(F_{K,\sigma}(u))_{K\in\mesh,\,\sigma\in\mathcal E_K}$ defined for
$u\in H_\disc$ ($H_\disc = X_{\disc,\Gamma_1}$ for the Signorini problem, $H_\disc = X_{\disc,0}$
for the obstacle problem) through the formula \cite[Eq. (2.25)]{B5}
\begin{multline*}
\forall K \in \mesh,\;\forall v \in H_\disc :\\
 \sum_{\sigma \in \mathcal{E}_K}|\sigma| F_{K,\sigma}(u)
(v_K-v_\sigma)=|K|\Lambda_K \nabla_K u \cdot \nabla_K v
+R_K^{T}(u)\mathbb{B}_K R_{K}(v).
\end{multline*}
Note that the flux $F_{K,\sigma}$ thus constructed is an approximation of
$-\frac{1}{|\sigma|}\int_\sigma \Lambda\nabla \bar u\cdot\mathbf{n}_{K,\sigma}$.
The variational inequalities \eqref{HMM-ineq-obs} and \eqref{HMM-ineq-sig} can then be re-written 
in terms of the balance and conservativity of the fluxes in the following ways:
\begin{itemize}
\item Signorini problem:
\[
\begin{aligned}
\sum_{\sigma \in \mathcal{E}_K}|\sigma|F_{K,\sigma}(u)= m(K)f_K, &\quad \forall K \in \mesh,\\
F_{K,\sigma}(u)+F_{L,\sigma}(u)= 0, &\quad \forall \sigma\in\mathcal E_{\rm int}\mbox{ with }
\mesh_\sigma=\{K,L\},\\
u_\sigma = 0, &\quad \forall \sigma \in \mathcal{E}_{\rm ext} \mbox{ such that }\sigma\subset  \Gamma_1,\\
F_{K,\sigma}(u)=0, &\quad \forall K \in \mesh\,,\forall \sigma \in \mathcal{E}_K \mbox{ such that }\sigma\subset  \Gamma_2,\\
F_{K,\sigma}(u)(u_\sigma - a_\sigma)=0 , & \quad\forall K \in \mesh\,, \forall \sigma \in \mathcal{E}_K \mbox{ such that }\sigma\subset  \Gamma_3,\\
-F_{K,\sigma}(u)\leq 0, &\quad \forall K \in \mesh\,, \forall \sigma \in \mathcal{E}_K \mbox{ such that }\sigma\subset  \Gamma_3,\\
u_\sigma \leq a_\sigma,&\quad \forall \sigma \in \mathcal{E}_{\rm ext} \mbox{ such that }\sigma\subset \Gamma_{3} .
\end{aligned}
\]
Here, $a_\sigma$ is a constant approximation of $a$ on $\sigma$.

\item Obstacle problem:
\[
\begin{aligned}
\left(-\sum_{\sigma \in \mathcal{E}_K}|\sigma|F_{K,\sigma}(u)+ m(K)f_K\right)(g_K - u_K)=0 , &\quad \forall K \in \mesh,\\
F_{K,\sigma}(u)+F_{L,\sigma}(u)= 0, &\quad \forall \sigma\in\mathcal E_{\rm int}\mbox{ with }
\mesh_\sigma=\{K,L\},\\
\sum_{\sigma \in \mathcal{E}_K}|\sigma|F_{K,\sigma}(u)\leq m(K)f_K, &\quad \forall K \in \mesh,\\
u_K \leq g_K, &\quad \forall K \in \mesh,\\
u_\sigma=0, &\quad \forall \sigma \in \mathcal E_{\rm ext}.
\end{aligned}
\]
Here, $g_K$ is an approximation of $g$ on $K$, see Section \ref{sec:approxbarriers}.
\end{itemize}

The iteration monotony algorithm detailed in \cite{A2-23} can be naturally applied to these formulations.
We notice that by \cite[Section 4.2]{B-13} the HMM method is monotone when applied
to isotropic diffusion on meshes made of acute triangles; in that case, the convergence of the monotony algorithm can be established as in \cite{A2-23} for finite element and two-point finite volume methods. 
In our numerical tests, we noticed that, even when applied on meshes for which the monotony
of HMM is unknown or fails, the monotony algorithm actually still converges.

All tests given below are performed by MATLAB code using the PDE toolbox.


\subsection*{Test 1}
We investigate the Signorini problem from \cite{A28}:
\begin{align*}
-\Delta \bar u= 2\pi \sin (2\pi x) &\mbox{\quad in $\Omega=(0,1)^2$,} \nonumber\\
\bar{u} =0 &\mbox{\quad on $\Gamma_{1}$,} \nonumber\\
\nabla\bar{u}\cdot \mathbf{n}= 0 &\mbox{\quad on $\Gamma_{2}$,} \nonumber\\
\left.
\begin{array}{r}
\dsp \bar{u} \geq 0\\
\dsp\nabla\bar{u}\cdot \mathbf{n} \geq 0\\
\dsp\bar{u}\nabla\bar{u}\cdot \mathbf{n}= 0
\end{array} \right\} 
& \mbox{\quad on}\; \Gamma_{3},\nonumber
\end{align*}
with $\Gamma_1 = [0,1]\times \{1\}$, $\Gamma_2 = (\{0\}\times [0,1]) \cup (\{1\} \times [0,1])$ and $\Gamma_3=[0,1]\times \{0\}$.
Here, the domain is meshed with triangles produced by INITMESH.

Figure \ref{fig:A31} presents the graph of the approximate solution obtained
on a mesh of size $0.05$. The solution compares very well with linear finite elements solution from \cite{A-31}; the graph seems to perfectly capture the point where the condition on $\Gamma_3$ changes from Dirichlet to Neumann. Table \ref{tab:A31} shows the
number of iterations (NITER) of the monotony algorithm, required to obtain the HMM solution for various mesh sizes. Herbin in \cite{A2-23} proves that this number of iterations is theoretically bounded by the number of edges included in $\Gamma_3$. We observe that NITER is much less than this worst-case bound. We also notice the robustness of this monotony algorithm:
reducing the mesh size does not significantly affect the number of iterations.

\begin{figure}
	\begin{center}
	\includegraphics[scale=0.7]{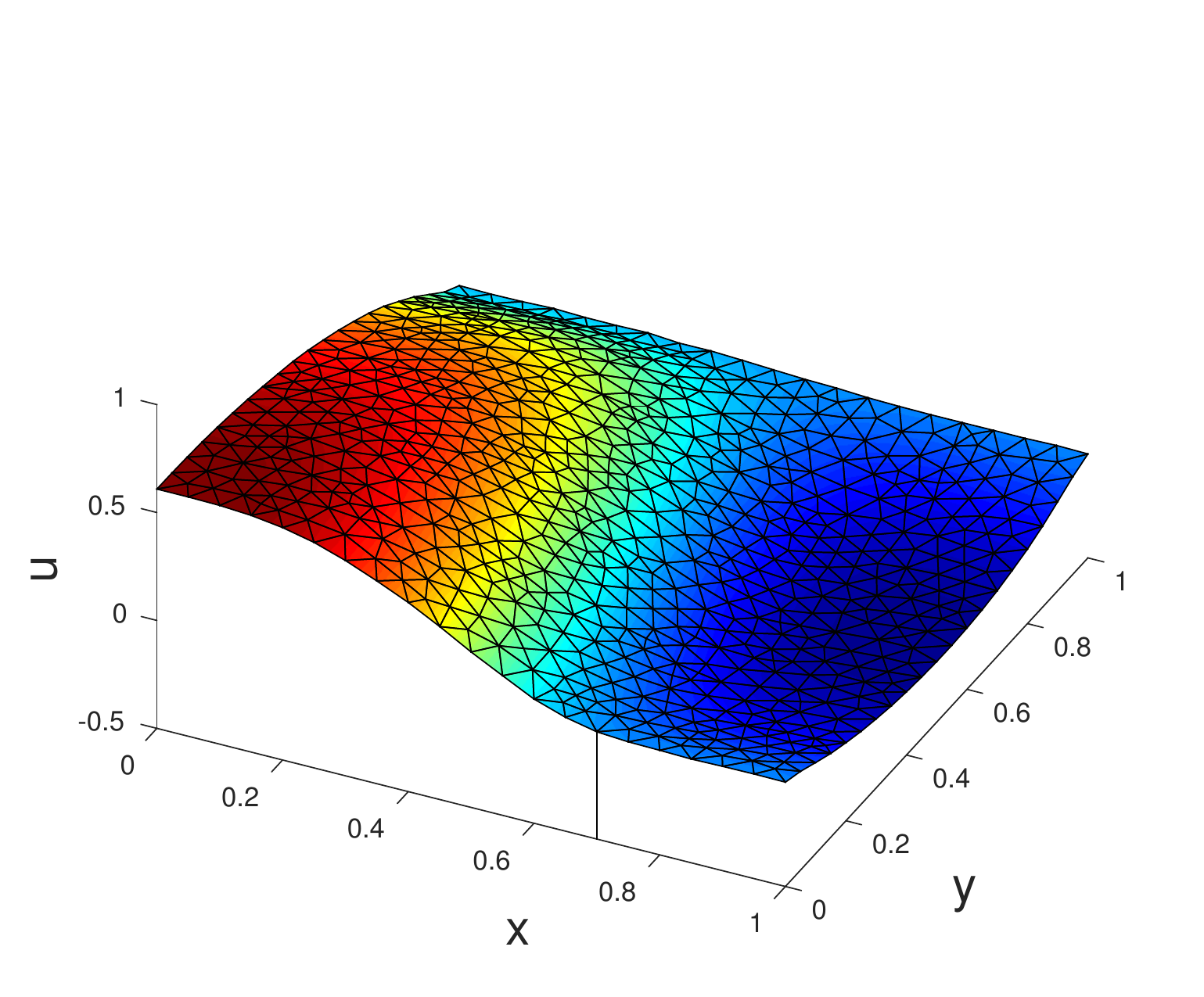}
	\end{center}
	\caption{The HHM solution for Test 1 with $h=0.05$. The line represents where the
boundary condition on $\Gamma_3$ changes from Neumann to Dirichlet.}
	\label{fig:A31}
\end{figure}

\begin{table}
\caption{Number of iterations of the monotony algorithm in Test 1}
\centering
\begin{tabular}{c rrrr}
\hline
Mesh size $h$ 
& 0.0625 & 0.0500 & 0.0250 & 0.0156 \\ 
\hline 
$ \#\{ \sigma \subset \Gamma_{3}\}$  
& 16 & 20 & 40 & 46 \\ 
NITER 
& 4 & 5 & 5 & 6 \\
\hline
\end{tabular}
\label{tab:A31}
\end{table}


\subsection*{Test 2} Most studies on variational inequalities, for instance \cite{A30,A28}, consider
test cases for which the exact solutions are not available; rates of convergence are then assessed by comparing
coarse approximate solutions with a reference approximate solution, computed on a fine mesh.
To more rigorously assess the rates of convergence,  we
develop here a new heterogeneous test case for Signorini boundary conditions, which has
an analytical solution with non-trivial one-sided conditions on $\Gamma_3$
(the analytical solution switches from
homogeneous Dirichlet to homogeneous Neumann at the mid-point of this boundary).

In this test case, we consider here \eqref{seepage1}--\eqref{sigcondition} with
the geometry of the domain presented in Figure \ref{fig-geomexact}, left.

\begin{figure}
\begin{tabular}{c@{\hspace*{3em}}c}
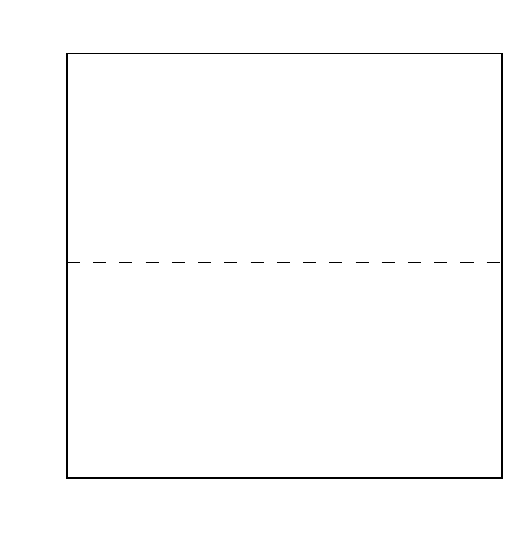 & \raisebox{1.5em}{\includegraphics[width=.35\linewidth]{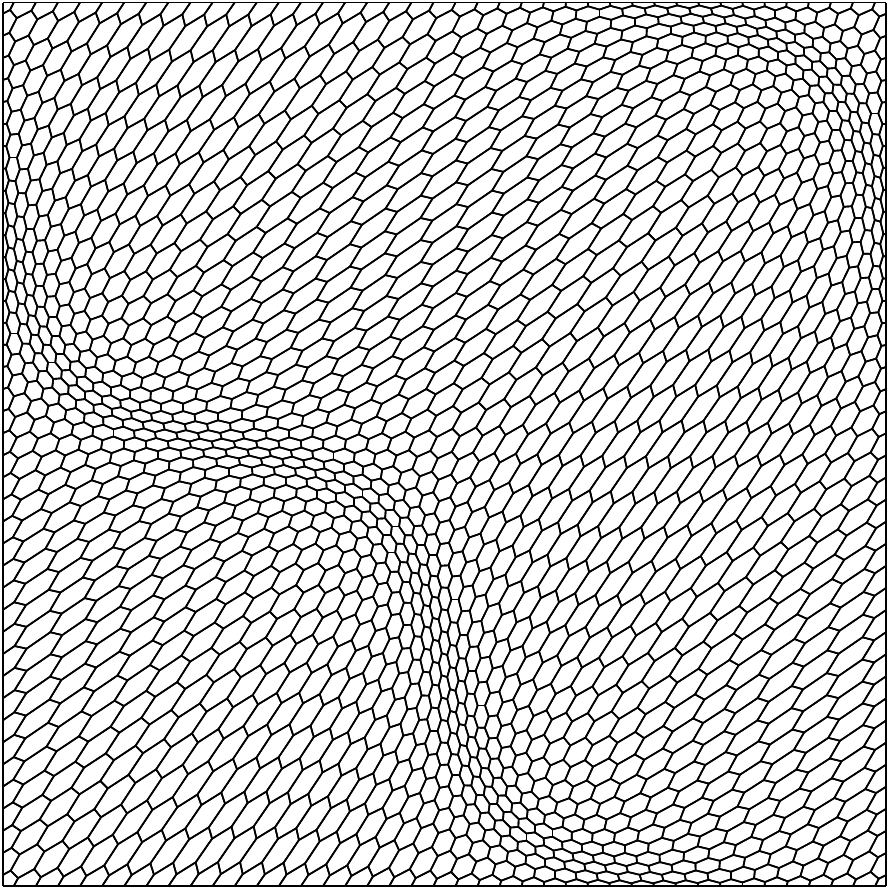}}
\end{tabular}
\caption{Test 2: geometry and diffusion (left), typical grid (right).}\label{fig-geomexact}
\end{figure}

The exact solution is
\begin{align*}
\bar u(x,y)=\left\{
\begin{array}{ll}
P(y)h(x) & \mbox{for $y \in (0,\frac{1}{2})$ and $x \in (0,1)$},\\
x g(x)G(y) & \mbox{for $y \in (\frac{1}{2},1)$ and $ x \in (0,1)$}.
\end{array} \right.
\end{align*}
To ensures that \eqref{seepage2} holds on the lower part of $\Gamma_1$,
and that the normal derivatives $\Lambda\nabla \bar u\cdot\mathbf{n}$
at the interface $y=\frac12$ match, we take $P$ such that $P(0)=P(\frac{1}{2})=P'(\frac{1}{2})=0$.
Assuming that $P<0$ on $(0,\frac12)$, the conditions $\partial_n \bar u = -\partial_x \bar u= 0$
and $\bar u < 0$ on the lower half of $\Gamma_3$,
as well as the homogeneous condition $\partial_n \bar u =0$ on the lower half of $\Gamma_2$,
will be satisfied if $h$ is such that $h'(0)=h'(1)=0$ and $h(0)=1$.
We choose
\[
P(y)=-y\left(y-\frac{1}{2}\right)^2\quad\mbox{ and }h(x)= \cos (\pi x).
\]
Given our choice of $P$, ensuring that $\div(\Lambda\nabla \bar u)\in L^2(\Omega)$
(i.e. that $\bar u$ and the normal derivatives match at $y=\frac12$) can be done by taking
$G$ such that $G(\frac12)=G'(\frac12)=0$. The boundary condition
\eqref{seepage2} on the upper part of $\Gamma_1$ is satisfied if
$G(1)=0$.
The boundary conditions $\partial_n \bar u=0$ on the upper part of $\Gamma_2$,
and $\bar u=0$ and $\partial_n \bar u<0$ on the upper part of $\Gamma_3$,
are enforced by taking $g$ such that $g(1)+g'(1)=0$ and $g(0)>0$
(with $G>0$ on $(\frac12,1)$). We select here 
\[
g(x)=\cos^2 \left(\frac{\pi}{2}x\right)=\frac{1+\cos(\pi x)}{2} \quad \mbox{and} \quad G(y)=(1-y)\left( y-\frac{1}{2}\right)^2.
\]
This $\bar u$ is then the analytical solution to the following problem
\begin{align*}
-\mathrm{div}(\Lambda\nabla\bar{u})= f &\mbox{\quad in $\Omega$,} \\
\bar{u} =0 &\mbox{\quad on $\Gamma_{1}$,}\\
\Lambda\nabla\bar{u}\cdot \mathbf{n}= 0 &\mbox{\quad on $\Gamma_{2}$,}\\
\left.
\begin{array}{r}
\dsp \bar{u} \leq 0\\
\Lambda\nabla\bar{u}\cdot \mathbf{n} \leq 0\\
\bar{u}\Lambda\nabla\bar{u}\cdot \mathbf{n}= 0
\end{array} \right\} 
& \mbox{\quad on}\; \Gamma_{3},
\end{align*}
with, for $y \le \frac{1}{2}$,
\[
f(x,y)=100\left[\pi^2y \cos(\pi x)\left(y - \frac{1}{2}\right)^2 - 2y \cos(\pi x)
- 2 \cos(\pi x)(2y - 1)\right],
\]
and, for $y >\frac{1}{2}$,
\begin{align*}
f(x,y)={}&\frac{\pi^2 x}{2}\cos(\pi x)(y - 1)\left(y - \frac{1}{2}\right)^2 -
2x\cos^2\left(\frac{\pi x}{2}\right)(3y - 2)\\
& + \pi \sin(\pi x)(y - 1)\left(y - \frac12\right)^2.
\end{align*}


\begin{figure}[ht]
	\begin{center}
	\includegraphics[scale=0.5]{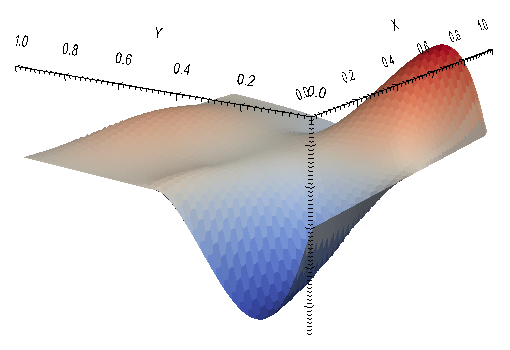}
	\end{center}
	\caption{The HHM solution for Test 2 on an hexagonal mesh with $h=0.07$.}
	\label{fig:approx}
\end{figure}
 
We test the scheme on a sequence of meshes (mosly) made of hexagonal cells.
The third mesh in this sequence (with mesh size $h=0.07$) is represented
in Figure \ref{fig-geomexact}, right.

The relative errors on $\bar u$ and $\nabla\bar u$, and the corresponding orders of convergence (computed from one mesh to the next one), are presented in Table \ref{tab:exact3}. For the HMM method, Theorem \ref{theorem:improveseepage} predicts a convergence of order 1 on the gradient if the solution 
belongs to $H^2$ and $\Lambda\nabla\bar u\in H^1$; the observed numerical rates are not slightly above
these values, probably due to the fact that $\Lambda$ is not Lipschitz-continuous here (and thus
the regularity  $\Lambda\nabla\bar u\in H^1$ is not satisfied).
As in the previous test, the number of iterations of the monotony algorithm
remains well below the theoretical bound.

\begin{table}[h]
\caption{Error estimate and number of iterations for Test 2} 
\centering
\begin{tabular}{c rrrr}
\hline
mesh size $h$ 
& 0.24 & 0.13 & 0.07 & 0.03\\ 
\hline 
$\frac{\| \bar u - \Pi_\disc u\|_{L^{2}(\O)}}{\|\bar u\|_{L^2(\O)}}$ 
& 0.6858 & 0.2531 & 0.1355 & 0.0758\\ 
Order of convergence 
&-- & 1.60 & 0.92 & 0.84 \\
\hline
$\frac{\| \nabla \bar u - \nabla_\disc u\|_{L^{2}(\O)}}{\|\nabla\bar u\|_{L^2(\O)}}$ 
& 0.4360 & 0.2038 & 0.1041 & 0.0542\\ 
Order of convergence 
&-- & 1.22 & 0.99 & 0.95\\
\hline
$\#\{\sigma \subset \Gamma_{3}\}$ 
& 20 & 40 & 80 & 160 \\
NITER 
& 5 & 5 & 6 & 7 \\
\hline
\end{tabular}
\label{tab:exact3}
\end{table}

The solution to the HMM scheme for the third mesh in the family used in Table \ref{tab:exact3}
is plotted in Figure \ref{fig:approx}.
On $\Gamma_3$, the saturated constraint for the exact solution changes from Neumann $\nabla\bar u\cdot\mathbf{n}=0$ to Dirichlet $\bar u=0$ at $y=0.5$. It is clear on Figure \ref{fig:approx} that the HMM scheme
captures well this change of constraint. The slight bump visible at $y=1/2$ is most
probably due to the fact that the mesh is not aligned with the heterogeneity of $\Lambda$
(hence, in some cells the diffusion tensor takes two -very distinct- values, and its approximation by one
constant value smears the solution).
  

When meshes are aligned with data heterogeneities, such bumps do not appear. This is illustrated
in Figure \ref{fig:kershaw}, in which we represent the solution obtained when using a ``Kershaw'' mesh as in
the FVCA5 benchmark \cite{HH08}. This mesh has size $h=0.16$, 34 edges on $\Gamma_3$,
and the monotony algorithm converges in 7 iterations. The
relative $L^2$ error on $\bar u$ and $\nabla\bar u$ are respectively $0.017$ and $0.019$.
As expected on these kinds of extremely distorted
meshes, the solution has internal oscillations, but is otherwise qualitatively good.
In particular, despite distorted cells near the boundary $\Gamma_3$, the
transition between the Dirichlet and the Neumann boundary conditions is well
captured.

\begin{figure}[ht]
	\begin{center}
	\begin{tabular}{c@{\hspace*{1em}}c}
	\includegraphics[width=.55\linewidth]{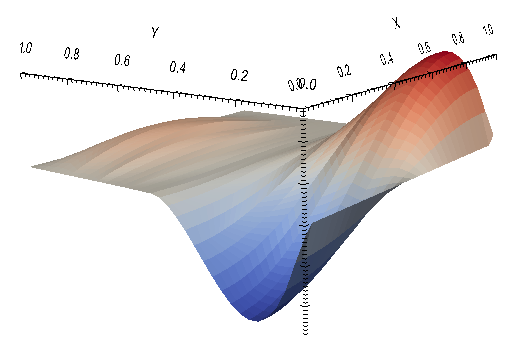} & \includegraphics[width=.4\linewidth]{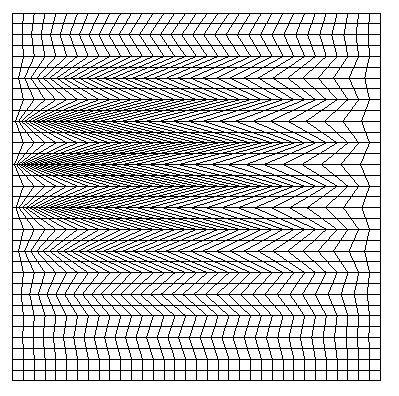}
	\end{tabular}
	\end{center}
	\caption{The HHM solution (left) for Test 2 on a Kershaw mesh (right).}
	\label{fig:kershaw}
\end{figure}



\subsection*{Test 3}
In this case test, we apply the HMM method to the homogeneous obstacle problem:  
\begin{align*}
(\Delta \bar u+C)(\psi - \bar{u}) = 0, &\quad\mbox{in $\Omega=(0,1)^2$,} \\
-\Delta \bar u\geq C,&\quad \mbox{in $\Omega$,} \\
\bar{u}\geq \psi, &\quad\mbox{in $\Omega$,}\\
\bar{u} = 0, &\quad\mbox{on $\partial\Omega$.}
\end{align*} 
The constant $C$ is negative, and the obstacle function $\psi(x,y)= - \min (x, 1-x, y, 1-y)$
satisfies $\psi =0$ on $\partial\O$.

\par Figure \ref{fig:A34} shows the graph of the HMM solution to the above obstacle problem with respect to $h=0.05$. Table \ref{tab:A34} illustrates the performance of the method and the algorithm. Here again, the number of iterations required to obtain the solution is far less than the number of cells, which is a theoretical bound in the case of obstacle problem \cite{A2-23}. Our results compare well with the results obtained by semi-iterative Newton-type methods in \cite{brugnano2009}, which indicate that decreasing $|C|$ contributes to the difficulty of the problem
(leading to an increase number of iterations). We note that, for a mesh of nearly 14,000 cells, we only require 29 iterations if $|C|=5$ and 14 iterations if $|C|=20$. On a mesh of 10,000 cells,
the semi-iterative Newton-type method of \cite{brugnano2009} requires 32 iterations
if $|C|=5$ and 9 iterations if $|C|=20$.

\begin{figure}[ht]
	\begin{center}
	\includegraphics[scale=0.5]{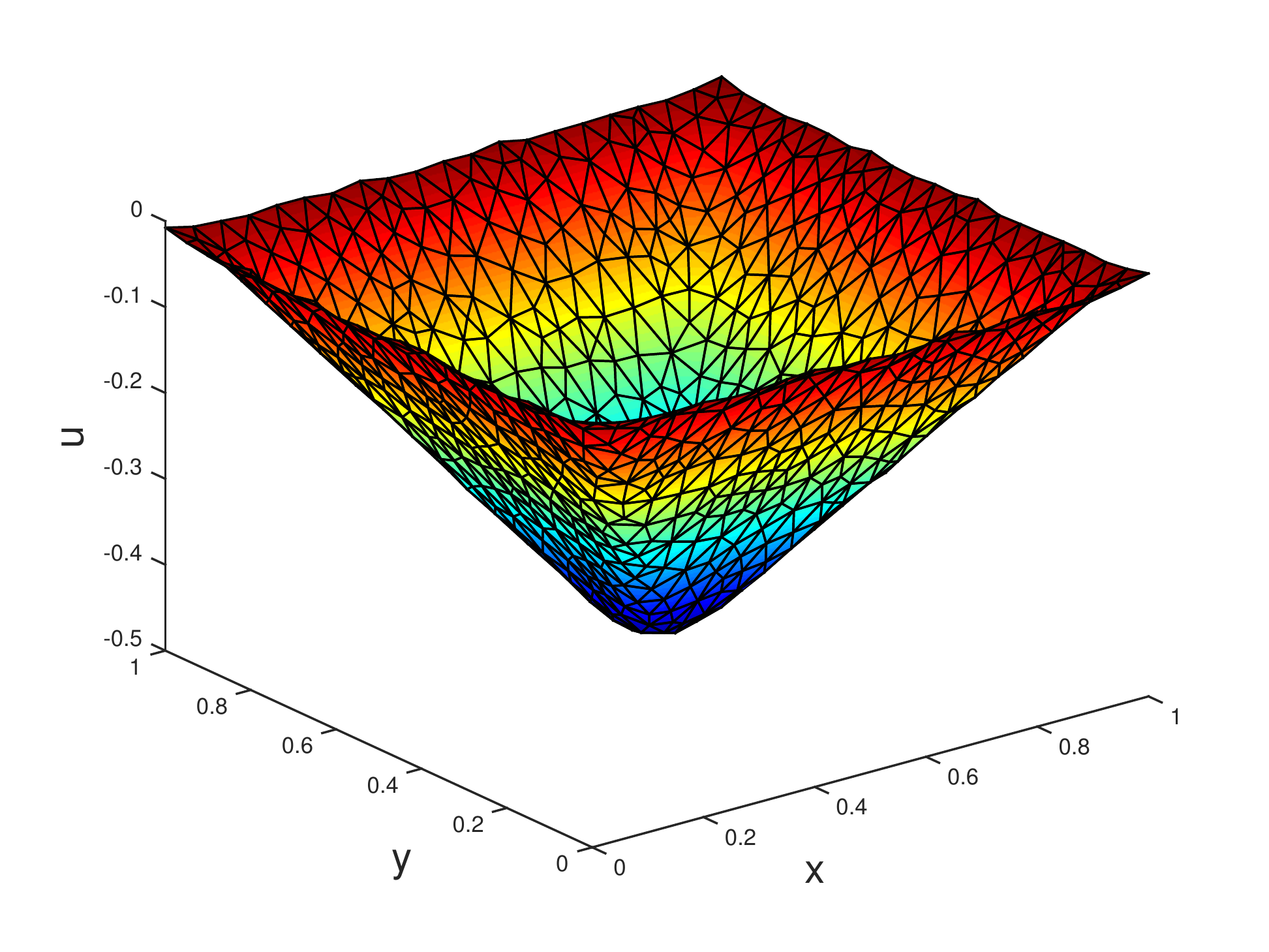}
	\end{center}
	\caption{The HHM solution for Test 3
	 with $h=0.05$, $C=-20$.}
	\label{fig:A34}
\end{figure}

\begin{table}[h]
\caption{The number of iterations for various $C$} 
\centering
\begin{tabular}{c rrrr}
\hline
mesh size $h$ 
& 0.016 & 0.025 & 0.050 & 0.062 \\ 
\hline 
$\# \mesh$
& 14006 & 5422 & 1342 & 872\\ 
NITER $(C=-5)$ 
& 29 & 20 & 12 & 11\\
NITER $(C=-10)$ 
&23 & 18 & 10 & 9\\
NITER $(C=-15)$ 
& 19 & 13 & 9 & 8\\
NITER $(C=-20)$ 
& 14 & 12 & 7 & 8\\
\hline
\end{tabular}
\label{tab:A34}
\end{table}



\section{Proofs of the theorems}
\label{proof}
\vskip .5pc
\noindent {\bf Proof of Theorem \ref{errorseepage}.} 
Since $\mathcal{K}_{\mathcal{D}}$ is a non-empty convex closed set, the existence and uniqueness of a solution to Problem (\ref{gsseepage})
follows from Stampacchia's theorem.

We note that $\Lambda\nabla\bar{u}\in H_{\div}(\O)$ since
\begin{equation}\label{seep:distr}
-\mathrm{div}(\Lambda\nabla\bar{u}) = f\mbox{ in the sense of distributions}
\end{equation}
(use $v=\bar u+\psi$ in \eqref{weseepage}, with $\psi\in C^\infty_c(\O)$). 
Taking $\bpsi=\Lambda\nabla\bar{u}$ in the definition \eqref{conformityseepage} of $W_\disc$, for any $z \in X_{\mathcal{D},\Gamma_{1}}$ we get
\begin{multline}
\int_{\Omega}\nabla_{\mathcal{D}} z \cdot \Lambda\nabla\bar{u}\ud x + \int_{\Omega}\Pi_{\mathcal{D}} z \cdot \mathrm{div}(\Lambda\nabla\bar{u}){\ud x} 
\leq\\
 \|\nabla_{\mathcal{D}} z\|_{L^{2}(\Omega)^{d}}
W_{\mathcal{D}}(\Lambda\nabla\bar{u})
+\langle \gamma_\bfn(\Lambda\nabla\bar u),
\mathbb{T}_{\mathcal{D}}z\rangle.
\label{seepagepr1}
\end{multline}
Let us focus on the last term on the right-hand side
of the above inequality. For any $v_{\mathcal{D}} \in \mathcal K_\disc$, we have
\begin{multline}
\langle \gamma_\bfn(\Lambda\nabla\bar{u}), \mathbb{T}_{\mathcal{D}}(v_{\mathcal{D}}-u) \rangle=\\
\langle \gamma_\bfn(\Lambda\nabla\bar{u}), \mathbb{T}_{\mathcal{D}}v_{\mathcal{D}}-\gamma\bar{u} \rangle
+\langle \gamma_\bfn(\Lambda\nabla\bar{u}),\gamma\bar{u}-\mathbb{T}_{\mathcal{D}}u\rangle.
\label{seepagepr2}
\end{multline}
The definition of the space $H_{\Gamma_{1}}^{1/2}(\partial\Omega)$ shows that there exists $w \in H^1(\Omega)$ such that $\gamma w = \mathbb{T}_{\mathcal{D}}u$ (this implies that $w\in \mathcal K$). According to the definition of the duality product \eqref{trace}, we have
\begin{equation*}
\langle \gamma_\bfn(\Lambda\nabla\bar{u}),\gamma\bar{u}- \mathbb{T}_{\mathcal{D}}u\rangle
=\int_{\Omega}\Lambda\nabla\bar{u}\cdot\nabla(\bar{u}-w)\ud x+\int_{\Omega}\mathrm{div}(\Lambda\nabla\bar{u})(\bar{u}-w) \ud x.
\end{equation*}
Since $\bar{u}$ satisfies \eqref{seep:distr}, it follows that
\begin{equation}\label{seepagepr3}
\langle \gamma_\bfn(\Lambda\nabla\bar{u}),\gamma\bar{u}- \mathbb{T}_{\mathcal{D}}u\rangle
=\int_{\Omega}\Lambda\nabla\bar{u}\cdot\nabla(\bar{u}-w)\ud x
-\int_{\Omega}f(\bar{u}-w) \ud x,
\end{equation}
which leads to, taking $w$ as a test function in the weak formulation \eqref{weseepage},
\begin{equation}\label{seepagepr3.1}
\langle \gamma_\bfn(\Lambda\nabla\bar{u}),\gamma\bar{u}- \mathbb{T}_{\mathcal{D}}u\rangle
\leq 0.\end{equation}
From this inequality, recalling \eqref{seep:distr} and \eqref{seepagepr2} and setting $z = v_{\mathcal{D}}- u \in X_{\mathcal{D},\Gamma_{1}}$ in (\ref{seepagepr1}), we deduce that
\begin{multline}\label{takeover}
\int_{\Omega}\nabla_{\mathcal{D}}(v_{\mathcal{D}} - u)\cdot \Lambda\nabla\bar{u}\ud x
+\int_{\Omega}f\Pi_{\mathcal{D}}(u-v_{\mathcal{D}}) \ud x
\leq\\
\|\nabla_{\mathcal{D}}(v_\disc - u)\|_{L^{2}(\Omega)^{d}}
W_{\mathcal{D}}(\Lambda\nabla\bar{u})+G_{\mathcal{D}}(\bar{u},v_{\mathcal{D}}).
\end{multline}
We use the fact that $u$ is the solution to Problem (\ref{gsseepage}) to bound from below the term involving $f$
and we get
\begin{multline*}
\int_{\Omega}\Lambda\nabla_{\mathcal{D}}(v_{\mathcal{D}} - u)\cdot(\nabla\bar{u}-\nabla_{\mathcal{D}}u) \ud x \leq\\
 \|\nabla_{\mathcal{D}}(v_{\mathcal{D}} - u)\|_{L^{2}(\Omega)^{d}}
W_{\mathcal{D}}(\Lambda\nabla\bar{u})+G_{\mathcal{D}}(\bar{u},v_{\mathcal{D}})^+
\end{multline*}
(note that we also used the bound $G_\disc\le G_\disc^+$). Adding and subtracting $\nabla_{\mathcal{D}}v_{\mathcal{D}}$ in $\nabla\bar u-\nabla_\disc u$, and using Cauchy-Schwarz's inequality, this leads to
\begin{multline*}
\underline{\lambda}\|\nabla_{\mathcal{D}}v_{\mathcal{D}} - \nabla_{\mathcal{D}}u\|_{L^{2}(\Omega)^{d}}^{2}
\leq \\
\|\nabla_{\mathcal{D}}v_{\mathcal{D}} - \nabla_{\mathcal{D}}u\|_{L^{2}(\Omega)^{d}}
\bigg(W_{\mathcal{D}}(\Lambda\nabla\bar{u})
+\overline{\lambda}S_{\mathcal{D}}(\bar{u},v_{\mathcal{D}})\bigg)
+G_{\mathcal{D}}(\bar{u},v_{\mathcal{D}})^+.
\end{multline*}
Applying Young's inequality gives
\begin{equation}
\|\nabla_{\mathcal{D}}v_{\mathcal{D}} - \nabla_{\mathcal{D}}u\|_{L^{2}(\Omega)^{d}} \leq
\sqrt{\frac{2}{\underline{\lambda}}G_{\mathcal{D}}(\bar{u},v_{\mathcal{D}})^+
+\frac{1}{\underline{\lambda}^{2}}\left[W_{\mathcal{D}}(\Lambda\nabla\bar{u})+\overline{\lambda}S_{\mathcal{D}}(\bar{u},v_{\mathcal{D}})\right]^{2}}.
\label{proof1obs}
\end{equation}
Estimate (\ref{errorseepage1}) follows from $ \| \nabla_{\mathcal{D}}u-\nabla\bar{u} \|_{L^{2}(\Omega)^{d}} \leq  \| \nabla_{\mathcal{D}}u - \nabla_{\mathcal{D}}v_{\mathcal{D}}\|_{L^{2}(\Omega)^{d}} + \| \nabla_{\mathcal{D}}v_{\mathcal{D}}-\nabla\bar{u} \|_{L^{2}(\Omega)^{d}} 
$ and $\sqrt{a+b}\leq\sqrt{a}+\sqrt{b}$. Using (\ref{coercivityseepage}) and \eqref{proof1obs}, we obtain
\begin{equation*}
\|\Pi_{\mathcal{D}}v_{\mathcal{D}}-\Pi_{\mathcal{D}}u\|_{L^{2}(\Omega)} \leq
C_{\mathcal{D}}\sqrt{\frac{2}{\underline{\lambda}}G_{\mathcal{D}}(\bar{u},v_{\mathcal{D}})^+
+\frac{1}{\underline{\lambda}^{2}}\left[W_{\mathcal{D}}(\Lambda\nabla\bar{u})+\overline{\lambda}S_{\mathcal{D}}(\bar{u},v_{\mathcal{D}})\right]^{2}
}.
\end{equation*}
By writing $\|\Pi_{\mathcal{D}}u-\bar{u}\|_{L^{2}(\Omega)} \leq 
\|\Pi_{\mathcal{D}}u-\Pi_{\mathcal{D}}v_{\mathcal{D}}\|_{L^{2}(\Omega)}
+ \|\Pi_{\mathcal{D}}v_{\mathcal{D}}-\bar{u}\|_{L^{2}(\Omega)}$, the above inequality shows that
\begin{multline*}
\|\Pi_{\mathcal{D}}u-\bar{u}\|_{L^{2}(\Omega)} \leq\\
C_{\mathcal{D}}\sqrt{\frac{2}{\underline{\lambda}}G_{\mathcal{D}}(\bar{u},v_{\mathcal{D}})^+
+\frac{1}{\underline{\lambda}^{2}}\bigg(W_{\mathcal{D}}(\Lambda\nabla\bar{u})+\overline{\lambda}S_{\mathcal{D}}(\bar{u},v_{\mathcal{D}})\bigg)^{2}
}
+S_{\mathcal{D}}(\bar{u},v_{\mathcal{D}}).
\end{multline*}
Applying $\sqrt{a+b}\leq\sqrt{a}+\sqrt{b}$ again, Estimate \eqref{errorseepage2} is obtained and the proof is complete.

\medskip

\noindent {\bf Proof of Theorem \ref{errorobs}.} 

The proof of this theorem is very similar as the proof of \cite[Theorem 1]{Y}. We however
give some details for the sake of completeness.

As for the Signorini problem, the existence and uniqueness of the solution to Problem \eqref{gsobs}
follows from Stampacchia's theorem. Let us now establish the error estimates.
Under the assumption that $\mathrm{div}(\Lambda\nabla \bar{u}) \in L^{2}(\Omega)$, we note that $\Lambda\nabla\bar{u} \in 
H_{\mathrm{div}}(\Omega)$. For any $w\in X_{\mathcal{D},0}$, using $\bpsi=\Lambda\nabla \bar{u}$ in the definition (\ref{conformityobs})
of $W_\disc$
therefore implies
\begin{equation}
\int_{\Omega}\nabla_{\mathcal{D}} w \cdot \Lambda\nabla\bar{u} \ud x + \int_{\Omega}\Pi_{\mathcal{D}} w \cdot \mathrm{div}(\Lambda\nabla\bar{u}) \ud x \leq \|\nabla_{\mathcal{D}} w\|_{L^{2}(\Omega)^{d}}
W_{\mathcal{D}}(\Lambda\nabla\bar{u}).
\label{8}\end{equation}
For any $v_{\mathcal{D}} \in \mathcal{K}_{\mathcal{D}}$, one has
\begin{multline}\label{th:error-init}
\int_{\Omega}\Pi_{\mathcal{D}}(u-v_{\mathcal{D}})\mathrm{div}(\Lambda\nabla\bar{u})\ud x=
\int_{\Omega}(\Pi_{\mathcal{D}}u-g)(\mathrm{div}(\Lambda\nabla\bar{u})+f) \ud x\\
+\int_{\Omega}(g-\Pi_{\mathcal{D}}v_{\mathcal{D}})(\mathrm{div}(\Lambda\nabla\bar{u})+f) \ud x
-\int_{\Omega}f\Pi_{\mathcal{D}}(u-v_{\mathcal{D}}) \ud x.
\end{multline}
It is well known that the solution to the weak formulation \eqref{weobs}
satisfies \eqref{obs2} in the sense of distributions (use test functions $v=\bar u-\varphi$
in \eqref{weobs}, with $\varphi\in C^\infty_c(\O)$ non-negative). Hence, under our
regularity assumptions, \eqref{obs2} holds a.e.\ and, since $u \in \mathcal K_\disc$, we obtain ${\displaystyle\int_{\Omega}(\Pi_{\mathcal{D}}u-g)(\mathrm{div}(\Lambda\nabla\bar{u})+f)\ud x} \leq 0$. Hence,
\begin{align*}
\int_{\Omega}\Pi_{\mathcal{D}}(u-v_{\mathcal{D}})&\mathrm{div}(\Lambda\nabla\bar{u})\ud x\\
\le& \int_{\Omega}(g-\Pi_{\mathcal{D}}v_{\mathcal{D}})(\mathrm{div}(\Lambda\nabla\bar{u})+f)\ud x
-\int_{\Omega}f\Pi_{\mathcal{D}}(u-v_{\mathcal{D}})\ud x\\
=& \int_{\Omega}(g-\bar{u})(\mathrm{div}(\Lambda\nabla\bar{u})+f)\ud x
 +\int_{\Omega}(\bar{u}-\Pi_{\mathcal{D}}v_{\mathcal{D}})(\mathrm{div}(\Lambda\nabla\bar{u})+f)\ud x\\
& -\int_{\Omega}f\Pi_{\mathcal{D}}(u-v_{\mathcal{D}}) \ud x.
\end{align*}
Our regularity assumptions ensure that $\bar{u}$ satisfies (\ref{obs1}). Therefore, by definition of $E_\disc$,
\begin{equation}\label{obsinequality}
\int_{\Omega}\varPi_{\mathcal{D}}(u-v_{\mathcal{D}})\mathrm{div}(\Lambda\nabla\bar{u}) \ud x
\le  E_\disc(\bar u,v_\disc)-\int_{\Omega}f\Pi_{\mathcal{D}}(u-v_{\mathcal{D}}) \ud x.
\end{equation}
From this inequality and setting $w = v_{\mathcal{D}}- u \in X_{\mathcal{D},0}$ in (\ref{8}), we obtain
\begin{multline*}
\int_{\Omega}\nabla_{\mathcal{D}}(v_{\mathcal{D}}- u)\cdot \Lambda\nabla\bar{u}\ud x +
\int_{\Omega}f\Pi_{\mathcal{D}}(u-v_{\mathcal{D}}) \ud x
\leq\\
\|\nabla_{\mathcal{D}}({v_\disc} - u)\|_{L^{2}(\Omega)^{d}}
W_{\mathcal{D}}(\Lambda\nabla\bar{u})+E_{\mathcal{D}}(\bar{u},v_{\mathcal{D}}).
\end{multline*}
The rest of proof can be handled in much the same way as proof of Theorem \ref{errorseepage},
taking over the reasoning from \eqref{takeover}. See also \cite{Y} for more details.

\medskip

\noindent {\bf Proof of Theorem \ref{theorem:improveseepage}.} 
We follow the same technique used in \cite{A2}. According to Remark \ref{rem:makesense} and due to the regularity on the solution, since
$\mathbb{T}_\disc v_\disc=\gamma(\bar u)=0$ on $\Gamma_1$ and
$\gamma_\bfn(\Lambda\nabla\bar u)=0$ on $\Gamma_2$, we can write
\begin{equation*}
G_{\mathcal{D}}(\bar{u},v_{\mathcal{D}})=
\int_{\Gamma_{3}}(\mathbb{T}_\disc v_{\mathcal{D}}-\gamma(\bar u))
\gamma_\bfn(\Lambda\nabla\bar{u})\ud x.
\end{equation*}

We notice first that the assumptions on $\bar u$ ensure that it is a solution of the Signorini problem in the strong sense.
Applying Theorem \ref{errorseepage} we have
\begin{multline}\label{to.plug}
\|\nabla_{\mathcal{D}}u - \nabla\bar{u}\|_{L^{2}(\Omega)^{d}}
+ \|\Pi_{\mathcal{D}}u - \bar{u}\|_{L^{2}(\Omega)}\\
\leq 
{C}\left(W_{\mathcal{D}}(\Lambda\nabla\bar{u}) 
 + S_{\mathcal{D}}(\bar{u},v_{\mathcal{D}})+ \sqrt{G_{\mathcal{D}}(\bar{u},v_{\mathcal{D}})^+} \right)
\end{multline}
{with $C$ depending only on $\underline{\lambda}$, $\overline{\lambda}$ and an upper bound of $C_\disc$.}
{Since $S_\disc(\bar u,v_\disc)\le C_1h_{\mathcal M}$ by assumption, it remains to} estimate the last term $G_{\mathcal{D}}(\bar{u},v_{\mathcal{D}})^+$. 
We start by writing 
\begin{align*}
G_\disc(\bar u,v_\disc)&=
\int_{\Gamma_3} (\mathbb{T}_\disc v_{\mathcal{D}}-a)
\gamma_\bfn(\Lambda\nabla\bar{u})\ud x
+\int_{\Gamma_3} (a-\gamma(\bar u))
\gamma_\bfn(\Lambda\nabla\bar{u})\ud x\\
&=\int_{\Gamma_3} (\mathbb{T}_\disc v_{\mathcal{D}}-a)
\gamma_\bfn(\Lambda\nabla\bar{u})\ud x\\
&=\sum_{\sigma\in\edges\,,\;\sigma\subset\Gamma_3}\int_\sigma (\mathbb{T}_\disc v_{\mathcal{D}}-a)\gamma_\bfn(\Lambda\nabla\bar{u})\ud x\\
&=:\sum_{\sigma\in\edges\,,\;\sigma\subset\Gamma_3} G_\sigma(\bar u,v_\disc),
\end{align*}
where the term involving $(a-\gamma(\bar u))
\gamma_\bfn(\Lambda\nabla\bar{u})$ has been eliminated by using \eqref{sigcondition}.
We then split the study on each $\edge$ depending on the cases: either (i) $\gamma\bar u<a$ a.e. on $\sigma$,
or (ii) ${\rm meas}(\{ y\in \sigma\,:\, \gamma\bar{u}(y)-a=0\})>0$ (where ${\rm meas}$ is
the $(d-1)$-dimensional measure).

In Case (i), we have $\gamma_\bfn(\Lambda\nabla\bar{u})=0$ in $\sigma$ since $\gamma\bar{u}$ satisfies \eqref{sigcondition}.
Hence, $G_\sigma(\bar u,v_\disc)=0$.

In Case (ii), let us denote $\nabla_\edge$ the tangential gradient to $\edge$, and
let us recall that, as a consequence of Stampacchia's lemma, if $w\in W^{1,1}(\edge)$
then $\nabla_\edge w=0$ a.e. on $\{y\in\edge\,:\,w(y)=0\}$ (here, ``a.e.'' is for the
measure ${\rm meas}$). Hence, with $w=\gamma\bar u-a$ we obtain at least one $y_{0}\in \sigma$ such that $(\gamma\bar{u}-a)(y_{0})=0$ and $\nabla_\edge(\gamma\bar{u}-a)(y_{0})=0$. Let $F$ be the
face of $\partial\Omega$ that contains $\sigma$. Using a Taylor's expansion along a path 
on $F$ between $y_0$ and $x_\sigma$, we deduce
\begin{equation}\label{taylor-exp}
|(\gamma\bar{u}-a)(x_\sigma)|\le L_F h_\edge^2,
\end{equation}
where $L_F$ depends on $F$ and on the Lipschitz constant of the tangential derivative
$\nabla_{\partial\O} (\gamma\bar{u}-a)$ on $F$
(the Lipschitz-continuity of $\nabla_{\partial\O} (\gamma\bar{u}-a)$
on this face follows from our assumption that $\gamma\bar{u}-a\in W^{2,\infty}(\partial\Omega)$).
Recalling that $||\mathbb{T}_\disc v_\disc-\gamma\bar u(x_\sigma)||_{L^2(\edge)}\le C_2h_\edge^2|\edge|^{1/2}$,
we infer that $||\mathbb{T}_\disc v_\disc-a||_{L^2(\edge)} \leq Ch_{\mathcal M}^{2}|\edge|^{1/2}$,
and therefore that $|G_\sigma(\bar{u},v_{\mathcal{D}})| \leq C
h_{\mathcal M}^{2}||\gamma_\bfn(\Lambda\nabla\bar{u})||_{L^2(\edge)}|\edge|^{1/2}$. Here, $C$
depends only on $\O$, $\bar u$ and $a$.

Gathering the upper bounds on each $G_\sigma$ and using the Cauchy--Schwarz inequality
gives $G_\disc(\bar u,v_\disc) \le Ch_{\mathcal M}^2 ||\gamma_\bfn(\Lambda\nabla\bar{u})||_{L^2(\dr\O)}$, with $C$ depending only on $\O$, $\Lambda$, $\bar u$ and $a$. This implies $G_\disc(\bar u,v_\disc)^+
\le C h_\mesh^2$, and the proof is complete by plugging
this estimate into \eqref{to.plug}.

\medskip

\noindent {\bf Proof of Theorem \ref{theorem:improveobse}.} The proof is very similar
to the proof of Theorem \ref{theorem:improveseepage}. As in this previous proof,
we only have to estimate $E_\disc(\bar u,v_\disc)$, which can be re-written as
\begin{align*}
E_\disc(\bar u,v_\disc)&=
\int_\O (\div(\Lambda\nabla\bar u)+f)(\bar u-g)\ud x
+\int_\O (\div(\Lambda\nabla\bar u)+f)(g-\Pi_\disc v_\disc)\ud x\\
&=\int_\O (\div(\Lambda\nabla\bar u)+f)(g-\Pi_\disc v_\disc)\ud x\\
&=\sum_{K\in\mathcal M}\int_K (\div(\Lambda\nabla\bar u)+f)(g-\Pi_\disc v_\disc)\ud x
=:\sum_{K\in\mathcal M}E_K(\bar u,v_\disc),
\end{align*}
where the term involving $(\div(\Lambda\nabla\bar u)+f)(\bar u-g)$ has been eliminated
using \eqref{obs1}. Each term $E_K(\bar u,v_\disc)$ is then estimated by considering
two cases, namely: (i) either $\bar u<g$ on $K$, in which case $E_K(\bar u,v_\disc)=0$
since $f+\div(\Lambda\bar u)=0$ on $K$, or (ii) $|\{y\in K\,:\,\bar u(y)=g\}|>0$,
in which case $E_K(\bar u,v_\disc)$ is estimated by using a Taylor expansion and
the assumption $||\Pi_\disc v_\disc -\bar u(x_K)||_{L^2(K)}\le C_2 h_K^2|K|^{1/2}$.


\begin{remark}\label{relax:gconst} If $(x_K)_{K\in\mathcal M}$ are the centres of gravity of the cells,
the assumption that $g$ is constant can be relaxed under additional regularity hypotheses.
Namely, if $F:=\div(\Lambda\nabla\bar u)+f\in H^1(K)$ and $g\in H^2(K)$ for any
cell $K$, then by letting $F_K$ and $g_K$ be the mean values on $K$ of $F$
and $g$, respectively, we can write
\begin{align*}
E_K(\bar u,v_\disc) =& \int_K F(g-\Pi_\disc v_\disc)\ud x\\
=& \int_K F(g(x_K)-\Pi_\disc v_\disc)\ud x + \int_K (F-F_K)(g-g_K)\ud x\\
&+ \int_K F(g_K-g(x_K))\ud x=:T_{1,K}+T_{2,K}+T_{3,K}.
\end{align*}
The term $T_{1,K}$ is estimated as in the proof (since $\bar u(x_K) -g(x_K)$ can be estimated using a
Taylor expansion about $y_0$). We estimate $T_{2,K}$ by using the Cauchy-Schwarz inequality
and classical estimates between an $H^1$ function and its average:
\[
|T_{2,K}|\le ||F-F_K||_{L^2(K)}||g-g_K||_{L^2(K)}\le C h_K^2 ||F||_{H^1(K)}||g||_{H^1(K)}.
\]
As for $T_{3,K}$, we use the fact that $g\in H^2(\O)$ and that $x_K$ is the center of gravity
of $K$ to write $|g(x_K)-g_K|\le C h_K^2||g||_{H^2(K)}$.
Combining all these estimates to bound $E_K(\bar u,v_D)$ and using Cauchy-Schwarz inequalities
leads to an upper bound in $\mathcal O(h_{\mathcal M}^2)$ for $E_\disc(\bar u,v_D)$.
\end{remark}



\section{The case of approximate barriers}\label{sec:approxbarriers}

For general barrier functions ($a$ in the Signorini problem, $g$ in the obstacle problem),
it might be challenging to find a proper approximation of $\bar u$
inside $\mathcal K_\disc$. Consider for example the $\mathbb{P}_1$ finite element method; the approximation is
usually constructed using the values of $\bar u$ at the nodes of the mesh, which only ensures
that this approximation is bounded above by the barrier at these nodes, not necessarily everywhere
else in the domain. It is therefore usual to relax the upper bound imposed on the solution to the numerical
scheme, and to consider only approximate barriers in the schemes.


\subsection{The Signorini problem}
We consider the gradient scheme \eqref{gsseepage} with, instead of $\mathcal K_{\disc}$, the following convex set:
\begin{equation}
\widetilde{\mathcal{K}}_{\disc} = \{ v \in X_{\disc,\Gamma_{1}}\;:\; \mathbb{T}_{\disc}v \leq a_{\disc} \; \mbox{on}\; \Gamma_3 \},
\end{equation}
where $a_{\disc} \in L^2(\partial\O)$ is an approximation of $a$.
The following theorems state error estimates for this modified gradient scheme.

\begin{theorem}[Error estimates for the Signorini problem with approximate bar\-rier]
\label{errorseepage-approx}
~
Under the assumptions of Theorem \ref{errorseepage}, if $\widetilde{\mathcal K}_\disc$ is not empty then there exists
a unique solution $u$ to the gradient scheme \eqref{gsseepage} in which $\mathcal K_\disc$ has been replaced with $\widetilde{\mathcal K}_\disc$.
Moreover, if $a-a_\disc\in H^{1/2}_{\Gamma_1}(\partial\O)$, Estimates \eqref{errorseepage1} and \eqref{errorseepage2} hold for any $v_\disc \in \widetilde{\mathcal K}_\disc$,
provided that $G_{\disc}(\bar{u},v_{\disc})$ is replaced with
\[
\widetilde{G}_{\disc}(\bar u,v_\disc)= G_\disc(\bar u,v_\disc)+ 
\langle \gamma_\bfn(\Lambda\nabla\bar{u}), a-a_{\disc}\rangle
\]
\end{theorem}

\begin{proof}
The proof is identical to that of Theorem \ref{errorseepage}, provided we can control the left-hand
side of \eqref{seepagepr2}. For any $v_{\disc} \in \widetilde{\mathcal{K}}_{\disc}$, we write
\begin{multline*}
\langle \gamma_\bfn(\Lambda\nabla\bar{u}), \mathbb{T}_{\mathcal{D}}(v_{\mathcal{D}}-u) \rangle=
\langle \gamma_\bfn(\Lambda\nabla\bar{u}), \mathbb{T}_{\mathcal{D}}v_{\mathcal{D}}-\gamma\bar{u} \rangle\\
+\langle \gamma_\bfn(\Lambda\nabla\bar{u}),\gamma\bar{u}-(\mathbb{T}_{\mathcal{D}}u+a-a_{\disc})\rangle
+\langle \gamma_\bfn(\Lambda\nabla\bar{u}), a-a_{\disc}\rangle.
\end{multline*}
We note that the first term in the right-hand side is $G_\disc(\bar u,v_\disc)$;
hence, the first and last terms in this right-hand side sums up to
$\widetilde{G}_\disc(\bar u,v_\disc)$ and we only need to prove that
the second term is non-positive. We take $w \in H^1(\O)$ such that $\gamma(\bar w)= \mathbb{T}_{\disc}u+a-a_{\disc}$, and we notice that $w$ belongs to $\mathcal K$
since $\mathbb{T}_\disc u$ and $a-a_\disc$ both belong to $H^{1/2}_{\Gamma_1}(\partial\O)$,
and since $\mathbb{T}_\disc u+a-a_\disc\le a$ on $\Gamma_3$. Similarly to
\eqref{seepagepr3}, the definition of $w$ shows that
\[
\langle \gamma_\bfn(\Lambda\nabla\bar{u}),\gamma\bar{u}-(\mathbb{T}_{\mathcal{D}}u+a-a_{\disc})\rangle
=\int_\O \Lambda\nabla\bar u\cdot\nabla(\bar u-w)\ud x-\int_\O f(\bar u-w)\ud x.
\]
Hence, by \eqref{weseepage}, $\langle \gamma_\bfn(\Lambda\nabla\bar{u}),\gamma\bar{u}-(\mathbb{T}_{\mathcal{D}}u+a-a_{\disc})\rangle\le 0$ and, as required, the left-hand side
of \eqref{seepagepr2} is bounded above by $\widetilde{G}_\disc(\bar u,v_\disc)$.
\end{proof}

In the case of the $\mathbb{P}_1$ finite element method, $a_\disc$ is the $\mathbb{P}_1$ approximation of $a$ on $\dr\O$ constructed from its values
at the nodes. The interpolant $v_\disc$ of $\bar u$ mentioned in Section \ref{sec:improvesig} is bounded above
by $a=a_\disc$ at the nodes, and therefore everywhere by $a_\disc$; $v_\disc$ therefore belongs to $\widetilde{\mathcal K}_\disc$
and can be used in Theorem \ref{errorseepage-approx}. Under $H^2$ regularity of $a$, we have
$||a_\disc-a||_{L^2(\dr\O)}=\mathcal O(h_{\mathcal M}^2)$. Hence, if $\gamma_\bfn(\Lambda\nabla\bar u)\in L^2(\dr\O)$, we see that $\widetilde{G}_\disc(\bar u,v_\disc)=\mathcal O(h_{\mathcal M}^2)$
and Theorem \ref{errorseepage-approx} thus gives an order one estimate.

For several low-order methods (e.g. HMM, see \cite{B1}), the interpolant $v_\disc$ of the exact
solution $\bar u$ is constructed such that $\mathbb{T}_\disc v_\disc=\gamma(\bar u)(\centeredge)$ on each
$\edge\in\edgesext$.
The natural approximate barrier is then piecewise constant, and an order one error estimate
can be obtained as shown in the next result.

\begin{theorem}[Signorini problem: order of convergence for piecewise constant reconstructions]
Under Hypothesis \ref{hyseepage}, let $\disc$ be a gradient discretisation, in the sense of Definition \ref{gdseepage}, and let $\polyd=(\mesh,\edges,\centers)$ be a polytopal
mesh of $\O$. Let $a_\disc\in L^2(\partial\O)$ be
such that $a_\disc-a=0$ on $\Gamma_1$ and, for any $\sigma\in\edgesext$ such that
$\sigma\subset \Gamma_3$, $a_\disc=a(x_\sigma)$ on $\sigma$, where $x_\sigma\in\sigma$.
Let $\bar u$ be the solution to \eqref{weseepage}, and let us assume that
$\gamma_\bfn(\Lambda\nabla\bar u)\in L^2(\partial\O)$ and that $\gamma(\bar u)-a\in W^{2,\infty}(\partial\O)$.
We also assume that there exists an interpolant $v_{\mathcal{D}} \in X_{\disc,\Gamma_{1}}$
such that $S_{\mathcal{D}}(\bar{u},v_{\mathcal{D}})\le C_1h_{\mathcal M}$ with $C_1$ not depending on $\disc$ or $\polyd$, and, for any $\sigma\subset \Gamma_3$, $\mathbb T_{\mathcal{D}}v_{\mathcal{D}}=\gamma(\bar{u})(x_{\sigma})$ on $\sigma$.

Then $\widetilde{\mathcal K}_\disc\not=\emptyset$ and, if $u$ is the solution to \eqref{gsseepage} with $\mathcal K_\disc$
replaced with $\widetilde{\mathcal K}_\disc$, it holds
\begin{equation}
\|\nabla_{\mathcal{D}}u - \nabla\bar{u}\|_{L^{2}(\Omega)^{d}}
+ \|\Pi_{\mathcal{D}}u - \bar{u}\|_{L^{2}(\Omega)}
\leq Ch_{\mathcal M}+CW_{\disc}(\Lambda\nabla\bar u)
\label{eq:improverrorseepage2}\end{equation}
where $C$ depends only on $\O$, $\Lambda$, $\bar u$, $a$, $C_1$ and an upper bound of $C_\disc$.
\label{th:improvedapproxseepage}
\end{theorem}

\begin{proof}
Clearly, $v_{\disc}\in \widetilde{\mathcal K}_\disc$. The conclusion follows from Theorem \ref{errorseepage-approx}
if we prove that $\widetilde{G}_\disc(\bar u,v_\disc)^+=\mathcal O(h_{\mathcal M}^2)$. Note that, following Remark \ref{rem:makesense}, it makes sense to consider a piecewise
constant reconstructed trace $\mathbb{T}_\disc v$.
 
With our choices of $a_\disc$ and $\Pi_\disc v_\disc$, and using the fact that $\bar u$ is the solution to the Signorini problem in the strong sense (and satisfies therefore \eqref{sigcondition}), we have
\begin{align*}
\widetilde{G}_\disc(\bar u,v_\disc)&=
\int_{\Gamma_3} (\mathbb T_\disc v_\disc-\gamma\bar u)\gamma_\bfn(\Lambda\nabla\bar{u})\ud x
+\int_{\Gamma_3} (a-a_{\disc})\gamma_\bfn(\Lambda\nabla\bar{u})\ud x\\
&=
\int_{\Gamma_3} \gamma_\bfn(\Lambda\nabla\bar{u})(\mathbb T_\disc v_\disc-a_\disc)\ud x\\
&=\sum_{\sigma\in\edgesext,\,\sigma\subset \Gamma_3}\int_\sigma (\gamma\bar u (x_{\sigma})-a(x_\edge))\gamma_\bfn(\Lambda\nabla\bar{u})\ud x\\
&=:\sum_{\sigma\in\edgesext,\,\sigma\subset \Gamma_3}\widetilde G_\sigma(\bar u,v_\disc).
\end{align*}
We then deal edge by edge, considering two cases as in the proof of Theorem \ref{theorem:improveseepage}.
In Case (i), where $\gamma\bar u<a$ a.e. on $\sigma$, we have $\widetilde G_\sigma(\bar{u},v_{\mathcal{D}})=0$ since $\gamma_\bfn(\Lambda\nabla\bar{u})=0$ in $\sigma$. In Case (ii), we estimate $\widetilde G_\sigma(\bar u,v_\disc)$
using the Taylor expansion \eqref{taylor-exp}.
\end{proof}


\subsection{The obstacle problem}
\label{sec:approxbarriers obs}
With $g_\disc\in L^2(\O)$ an approximation of $g$, we consider the new convex set
\begin{equation}
\widetilde{\mathcal{K}}_{\disc} = \{ v \in X_{\disc,0}\;:\; \Pi_{\disc}v \leq g_{\disc} \; \mbox{in}\; \O \}
\end{equation}
and we write the gradient scheme \eqref{gsobs} with $\widetilde{\mathcal K}_\disc$ instead of $\mathcal K_\disc$.
The following theorems are the equivalent for the obstacle problem of
Theorems \ref{errorseepage-approx} and \ref{th:improvedapproxseepage}.

\begin{theorem}[Error estimates for the obstacle problem with approximate bar\-rier]
\label{errorobs-approx}
~
Under the assumptions of Theorem \ref{errorobs}, if $\widetilde{\mathcal K}_\disc$ is not empty then there exists
a unique solution $u$ to the gradient scheme \eqref{gsobs} in which $\mathcal K_\disc$ has been replaced with $\widetilde{\mathcal K}_\disc$.
Moreover, Estimates \eqref{errorobs1} and \eqref{errorobs2} hold for any $v_\disc \in \widetilde{\mathcal K}_\disc$,
provided that $E_{\disc}(\bar{u},v_{\disc})$ is replaced with
\[
\widetilde{E}_{\disc}(\bar u,v_\disc)= E_\disc(\bar u,v_\disc)+ \int_{\O} (\div(\Lambda\nabla\bar{u})+f)(g_\disc-g) \ud x.
\]
\end{theorem}

\begin{proof} We follow exactly the proof of Theorem \ref{errorobs}, except that we introduce $g_\disc$ instead
of $g$ in \eqref{th:error-init}. The first term in the right-hand side of this equation is then still bounded above by $0$, and the
second term is written
\begin{multline*}
\int_\O (g_\disc-\Pi_\disc v_\disc)(\div(\Lambda\nabla \bar u)+f)\ud x
=\int_\O (g_\disc-g)(\div(\Lambda\nabla \bar u)+f)\ud x\\
+\int_\O (g-\Pi_\disc v_\disc)(\div(\Lambda\nabla \bar u)+f)\ud x.
\end{multline*}
The first term in this right-hand side corresponds to the additional term in $\widetilde{E}_\disc(\bar u,v_\disc)$,
whereas the second term is the one handled in the proof of Theorem \ref{errorobs}.
\end{proof}

\begin{theorem}[Obstacle problem: order of convergence for piecewise constant reconstructions]
Let $\disc$ be a gradient discretisation, in the sense of Definition \ref{gdobs},
and let $\polyd=(\mesh,\edges,\centers)$ be a polytopal mesh of $\O$. Let $g_\disc\in L^2(\O)$ be defined by
\[
\forall K\in\mesh\,,\; g_\disc=g(x_K)\mbox{ on $K$}.
\]
Let $\bar u$ be the solution to \eqref{weobs}, and let us assume that
$\div(\Lambda\nabla\bar u)\in L^2(\O)$ and that 
$\bar u-g \in W^{2,\infty}(\O)$.
We also assume that there exists an interpolant $v_{\mathcal{D}} \in X_{\disc,0}$
such that $S_{\mathcal{D}}(\bar{u},v_{\mathcal{D}})\le C_1h_{\mathcal M}$, with $C_1$ not depending on $\disc$ or $\polyd$, and that
$\Pi_{\mathcal{D}}v_{\mathcal{D}}=\bar{u}(x_{K})$ on $K$, for any $K\in\mesh$.

Then $\widetilde{\mathcal K}_\disc\not=\emptyset$ and, if $u$ is the solution to \eqref{gsobs} with $\mathcal K_\disc$
replaced with $\widetilde{\mathcal K}_\disc$,
\begin{equation}
\|\nabla_{\mathcal{D}}u - \nabla\bar{u}\|_{L^{2}(\Omega)^{d}}
+ \|\Pi_{\mathcal{D}}u - \bar{u}\|_{L^{2}(\Omega)}
\leq Ch_{\mathcal M}+CW_{\disc}(\Lambda\nabla\bar u)
\label{eq:improverrorobs2}\end{equation}
where $C$ depends only on $\O$, $\Lambda$, $\bar u$, $g$, $C_1$ and an upper bound of $C_\disc$.
\label{th:improvedapproxobs}\end{theorem}

\begin{proof}
The proof can be conducted by following the same ideas as in the proof of Theorem \ref{th:improvedapproxseepage}.
\end{proof}

\medskip

Let us now compare our results with previous studies. In \cite{C3,F1},
$\mathcal O(h_\mesh)$ error estimates are established
for $\mathbb{P}_1$ and mixed finite elements applied to the Signorini problem and the obstacle problem. 
This order was obtained under the assumptions that $\bar u \in W^{1,\infty}(\O) \cap H^{2}(\O)$ and $a$ is constant
for the Signorini problem and that $\Lambda\equiv I_{d}$, that
$\bar u$ and $g$ are in $H^{2}(\O)$ for the case of the obstacle problem. 
Our results generalise these orders of convergence to the case of a Lipschitz-continuous $\Lambda$
and a non-constant $a$ (note that mixed finite elements are also part of the gradient schemes
framework, see \cite{B10,S1}).

Studies of non-conforming methods for variational inequalities are more scarce.
We cite \cite{A30}, that applies Crouzeix--Raviart methods to the obstacle problem and obtains an order $\mathcal O (h_\mesh)$ under strong regularity assumptions, namely $f \in L^{\infty}(\O)$, $\bar u -g \in W^{2,\infty}(\O)$, $g \in H^{2}(\O)$, $\Lambda\equiv I_{d}$ and the free boundary has a finite length. For the Signorini problem, under the assumptions that $a$ is constant, $\bar u \in W^{2,\infty}(\O)$ and $\Lambda\equiv I_{d}$, Wang and Hua \cite{A-31} give a proof of an order $\mathcal O (h_\mesh)$ for Crouzeix--Raviart methods.
The natural Crouzeix--Raviart interpolants are piecewise linear on the edges and the cells, and
if we therefore use it as $a_\disc$ for the Signorini problem, we see that under an $H^2$ regularity assumption on $a$,
we have $||a-a_\disc||_{L^2(\dr\O)}\le Ch_\mesh^2$.
Hence, the additional term in $\widetilde{G}_\disc(\bar u,v_\disc)$ in
Theorem \ref{errorseepage-approx} is of order $h_\mesh^2$ and this theorem therefore
gives back the known $\mathcal O(h_\mesh)$ order of convergence for the Crouzeix--Raviart
scheme applied to the Signorini problem. This is obtained under
slightly more general assumptions, since $\Lambda$ does not need to be ${\rm Id}$ here.
The same applies to the obstacle problem through Theorem \ref{theorem:improveobse}.

As shown in Section \ref{sec:examples}, our results also allowed us to obtain brand new orders of convergences, for methods not previously studied for (but relevant to) the Signorini or the obstacle problem.

We finally notice that most of previous research investigates the Signorini problem under the assumption that the barrier $a$ is a constant. On the contrary, Theorems \ref{errorseepage}, \ref{errorseepage-approx} and \ref{th:improvedapproxseepage} presented here are valid for non-constant barriers.

\section{Conclusion}

We developed a gradient scheme framework for two kinds of variational inequalities, the Signorini problem and the obstacle problem.
We showed that several classical conforming or non-conforming
numerical methods for diffusion equations fit into this framework,
and we established general error estimates and orders of convergence for all
methods in this framework, both for exact and approximate barriers.

These results present new, and simpler, proofs of optimal orders of convergence 
for some methods previously studied for variational inequalities.
They also allowed us to establish new orders of convergence for
recent methods, designed to deal with anisotropic heterogeneous
diffusion PDEs on generic grids but not yet studied for the obstacle or the Signorini problem.

We designed an HMM method for variational inequalities, and showed through numerical tests
(including a new test case with analytical solution) the efficiency of these
method.

\bigskip

\textbf{Acknowledgement}: the authors thank Claire Chainais-Hillairet for providing the initial MATLAB HMM code.



\section{Appendix }
\label{traceoperator}

We recall here some classical notions  on the normal trace of a function in $H_{\rm{div}}$. Let $\Omega \subset \mathbb{R}^{d}$ be a Lipschtiz domain. Then there exists a surjective trace mapping $\gamma: H^{1}(\Omega) \rightarrow H^{1/2}(\partial\Omega)$. 
For $\bpsi \in H_{\rm{div}}$, the normal trace $\gamma_{\bfn}(\bpsi) \in (H^{1/2}(\partial\Omega))'$  of $\bpsi$ is defined by 
\begin{equation}
\langle \gamma_{\bfn}(\bpsi), \gamma (w) \rangle=
\int_{\Omega}\bpsi\cdot\nabla w dx + \int_{\Omega}{\rm{div}}\bpsi\cdot w \ud x, \quad \mbox{for any}\; w \in H^{1}(\Omega), \label{trace}\end{equation}
where $\langle \cdot, \cdot\rangle$ denotes the duality product between $H^{1/2}(\partial\Omega)$ and $(H^{1/2}(\partial\Omega))'$. Stoke's formula in $H^1_0(\O)\times H_{\div}(\O)$ shows that this definition makes sense, i.e. the right-hand side is unchanged if we apply
to $\widetilde{w}\in H^1(\O)$ such that $\gamma(w)=\gamma(\widetilde{w})$.


\bibliographystyle{siam}
\bibliography{ref2}

\end{document}